\documentclass[a4paper,12pt]{amsart}

\makeindex

\usepackage{xcolor}
\usepackage[
colorlinks=true,
urlcolor=purple,
linkcolor=purple!87!black,
pdfborder={0 0 0}
]{hyperref}

\usepackage{amsfonts,graphics,amsmath,amsthm,amsfonts,amscd,amssymb,amsmath,latexsym,euscript, enumerate,kotex}
\usepackage{epsfig,url}
\usepackage{flafter}
\usepackage[all,cmtip,line]{xy}
\usepackage{array}
\usepackage[english]{babel}
\usepackage{overpic}
\usepackage{subfig}
\usepackage{multirow}
\usepackage{tikz-cd}
\usepackage{theoremref}

\usepackage{wrapfig}
\usepackage[shortlabels]{enumitem}
\setlist[enumerate, 1]{1\textsuperscript{o}}

\newtheorem{theorem}{Theorem}[section]
\newtheorem{lemma}[theorem]{Lemma}
\newtheorem{proposition}[theorem]{Proposition}
\newtheorem{conjecture}[theorem]{Conjecture}

\theoremstyle{definition} % italic or bold etc.
\newtheorem{definition}[theorem]{Definition}
\newtheorem{definition-lemma}[theorem]{Definition-Lemma}

\newtheorem{example}[theorem]{Example}

\theoremstyle{remark}
\newtheorem{remark}[theorem]{Remark}

\newtheorem{notation}[theorem]{Notation}
\numberwithin{equation}{section}

\newcommand{\C}{\mathbb{C}}
\newcommand{\R}{\mathbb{R}}
\newcommand{\Z}{\mathbb{Z}}

\newcommand{\Q}{\mathbb{Q}}

\def\P{\mathbb{P}}

\newcommand{\A}{\mathbb{A}}

\DeclareMathOperator{\ord}{ord}

\DeclareMathOperator{\Pic}{Pic}

\def\Pic{\operatorname{Pic}}

\def\Spec{\operatorname{Spec}}

\def\Supp{\operatorname{Supp}}

\def\codim{\operatorname{codim}}

\newcommand{\floor}[1]{\left\lfloor #1 \right\rfloor}
\newcommand{\ceil}[1]{\left\lceil #1 \right \rceil}
\let\oldframe\frame
\renewcommand\frame[1][allowframebreaks]{\oldframe[#1]}

\title[Asymptotically flat divisors and strongly $F$-regular type varieties]{Asymptotically flat divisors and strongly $F$-regular type varieties}

\date{\today}

\begin{document}

\author[D.~Kim]{Donghyeon Kim}
\address[Donghyeon Kim]{Department of Mathematics, Yonsei University, 50 Yonsei-ro, Seodaemun-gu, Seoul 03722, Republic of Korea}
\email{narimial0@gmail.com}

\begin{abstract}
In this paper, we define the notion of asymptotically flat divisor on a normal variety over $\C$, and prove that if $X$ is a strongly $F$-regular type variety and $K_X$ is asymptotically flat, then $X$ is of klt type.
\end{abstract}

\maketitle
\allowdisplaybreaks

\section{Introduction}

\begin{center}
\textit{In this paper, $(X,\Delta)$ is a \emph{couple} if $X$ is a normal variety and $\Delta$ is an effective $\Q$-Weil divisor. A couple $(X,\Delta)$ is a \emph{pair} if $K_X+\Delta$ is $\Q$-Cartier.}
\end{center}

\smallskip

The concept of strong $F$-regularity was first presented in \cite{HH90}, highlighting its significant relationship with singularities in the minimal model program, as further explored in \cite{HW02} and \cite{Tak04}. For instance, we observe the following:

\begin{theorem}[{cf. \cite[Theorem 3.3]{HW02} and \cite[Theorem 3.2]{Tak04}}] \label{HaraWatanabe}
Let $(X,\Delta)$ be an affine pair over $\C$. Then $(X,\Delta)$ is of strongly $F$-regular type (for the definition, see Definition \ref{str}) if and only if $(X,\Delta)$ is klt.
\end{theorem}

The concept of strong $F$-regularity can be defined in a couple setting, and hence, considering Theorem \ref{HaraWatanabe}, it is reasonable to suggest that the following notion suitably represents the notion of Kawamata log terminal singularities in a couple setting.

\begin{definition}
Let $(X,\Delta)$ be a couple. We say $(X,\Delta)$ is \emph{of klt type}\footnote{The term ``of klt type" is sometimes called ``potentially klt" in some texts, such as in \cite{Xu25}.} if there exists an effective $\Q$-Weil divisor $\Delta'$ on $X$ with the property that $(X,\Delta+\Delta')$ is klt.
\end{definition}

Note that if $(X,\Delta)$ is a pair, then $(X,\Delta)$ is klt if and only if $(X,\Delta)$ is of klt type.

\smallskip

One may consider expanding Theorem \ref{HaraWatanabe} to a couple setting. Hara and Watanabe conjectured that any strongly $F$-regular type couple is of klt type:

\begin{conjecture}[{cf. \cite[Problem 5.3.1]{HW02} and \cite[Question 7.1]{SS10}}] \label{conjecture}
Let $(X,\Delta)$ be a couple. Then $(X,\Delta)$ is of strongly $F$-regular type if and only if $(X,\Delta)$ is of klt type.
\end{conjecture}

Note that in \cite{CEMS18}, the authors proved Conjecture \ref{conjecture} under the assumption that 
$$\bigoplus_{m\ge 0}\mathcal{O}_X(\floor{-m(K_X+\Delta)})$$
is a finitely generated $\mathcal{O}_X$-algebra.

\smallskip

This paper introduces the concept of \emph{asymptotic flatness} (refer to Definition \ref{numerically} (a) for details). We demonstrate that in many cases, $\Q$-Weil divisors satisfy the asymptotic flatness condition.

\begin{proposition} \label{GG}
Let $X$ be a normal variety over $\C$, and let $D$ be a $\Q$-Weil divisor on $X$ such that either
\begin{itemize}
    \item[\emph{(a)}] $\bigoplus_{m\ge 0}\mathcal{O}_X(\floor{-mD})$ is a finitely generated $\mathcal{O}_X$-algebra,
    \item[\emph{(b)}] $D$ is numerically $\Q$-Cartier (cf. Definition \ref{numerically}), or
    \item[\emph{(c)}] $X$ has rational singularities and infinitely many reductions mod $p$ of $mD$ is $(S_3)$.
\end{itemize}
Then $D$ is asymptotically flat.
\end{proposition}

Here, the phrase \emph{infinitely many reductions mod $p$ of $mD$ are $(S_3)$} means that there exists a model $(X_R, D_R)$ of $(X, D)$ over a finitely generated $\Z$-algebra $R$, along with a Zariski-dense subset $S' \subseteq \Spec R$, such that for any positive integer $m$ and any $s \in S'$, the sheaf $\mathcal{O}_{X_{\overline{s}}}(mD_{\overline{s}})$ satisfies Serre's condition $(S_3)$ (refer to Example \ref{exex} for an instance of such a Weil divisor). It is noteworthy that when $X$ is of strongly $F$-regular type, then $X$ has rational singularities by \cite{Smi97}.

\smallskip

In particular, if $X$ is of klt type, then any Weil divisor is asymptotically flat by (a) in Proposition \ref{GG} and \cite{BCHM10}. %Note that it is highly nontrivial whether $D$ is $\Q$-Cartier if $D$ is $\Q$-Cartier in infinitely many reductions mod $p$ (cf. Remark \ref{diff}).

\smallskip

The main theorem of the paper provides a partial affirmative answer to Conjecture \ref{conjecture}.

\begin{theorem} \label{5000}
Let $(X,\Delta)$ be an affine couple over $\C$ with $K_X+\Delta$ asymptotically flat. Then $(X,\Delta)$ is of strongly $F$-regular type if and only if $(X,\Delta)$ is of klt type.
\end{theorem}

Note that our result partially extends \cite[Theorem D]{CEMS18}. Let us explain the argument in proving Theorem \ref{5000}. As noted in \cite{SS10}, the challenge in establishing Theorem \ref{5000} lies in the fact that, even if we show for a model $(X_R,\Delta_R)$ over a finitely generated $\Z$-algebra $R$, there exists an effective $\Q$-Weil divisor $\Delta'_s$ on $X_{\overline{s}}$ such that $(X_{\overline{s}},\Delta_{\overline{s}}+\Delta'_s)$ is klt, the construction of $\Delta'_s$ is largely characteristic-dependent, and thus we cannot lift $\Delta'_s$ to $X$. To tackle this issue, we will employ de Fernex-Hacon's characterization of klt type varieties as presented in \cite{dFH09}, providing a boundary-free notion of klt type variety.

\smallskip

However, if we try to prove Theorem \ref{5000}, then we encounter the situation that we must consider all resolutions of $(X,\Delta,\mathcal{O}_X(-m(K_X+\Delta)))$ for all $m$, and there is no guarantee that for some $p$, we can take reduction mod $p$ to all such resolutions.

\smallskip

Observing that $\left\{\mathcal{O}_X(-m(K_X+\Delta))\right\}_{m(K_X+\Delta)\text{ is Weil}}$ forms a graded sequence of fractional ideals on $X$, we might consider the Jonsson-Mustață conjecture (see $\mathfrak{q}=0$ setting of \cite[Conjecture 7.4]{JM12} and \cite[Theorem 1.1]{Xu20}). This theorem suggests that a log smooth model defining a quasi-monomial valuation could potentially act as a log resolution for all $(X,\Delta,\mathcal{O}_X(-m(K_X+\Delta)))$. Note that this idea starts in \cite{BLX22} to prove $K$-semistability is an open condition. Though \cite{BLX22} does not use the Jonsson-Mustață conjecture explicitly, the proof is almost the same as that of \cite[Theorem 1.1]{Xu20}, and we can view \cite[Theorem 1.1 and Corollary 1.2]{BLX22} as applications of the Jonsson-Mustață conjecture. The idea is first applied in solving problems related to positive characteristic in \cite{Kim25b}. Nonetheless, we are unable to directly employ the Jonsson-Mustață conjecture here, due to the Jonsson-Mustață conjecture that is presented in \cite{Xu20} can be used in the klt setting; we lack information about whether $(X,\Delta)$ is of klt nature.

\smallskip

Therefore, we pass $(X,\Delta)$ to a small $\Q$-factorialization (refer to Definition \ref{Q}) and then apply the Jonsson-Mustață conjecture to the small $\Q$-factorialization. We are able to establish that $A_{X,\Delta}(E)>0$ (see Definition \ref{see} for details) for each prime divisor $E$ over $X$ under the assumption that $K_X+\Delta$ is asymptotically flat, which is sufficient to confirm the existence of a small $\Q$-factorialization $X'$ of $(X,\Delta)$ and to define a graded series of ideals in $\mathcal{O}_{X'}$ that detects the klt property of $(X,\Delta)$. The remaining task is to use the Jonsson-Mustață conjecture and apply a Diophantine approximation as in \cite[Theorem 1.3]{KK23} and \cite[Proposition 4.2]{CJKL25}.

\smallskip

This paper is organized as follows: In Section \ref{22}, we collect basic definitions. In Section \ref{333}, we recall the definition of natural valuation and natural pullback from \cite{dFH09}, and prove some lemmas that will be used later. In Section \ref{4444}, we define log discrepancy for a couple $(X,\Delta)$, and prove that if $(X,\Delta)$ is a strongly $F$-regular type couple, then $A_{X,\Delta}(E)>0$ for every prime divisor $E$ over $X$ under the assumption that $K_X+\Delta$ is asymptotically flat. In Section \ref{55555}, we prove the main theorem, Theorem \ref{5000}.

\section{Preliminaries} \label{22}
In this paper, we define \emph{variety} as an integral, separated, and finite-type scheme over an algebraically closed field. For convenience, in this paper, all varieties are assumed to be quasi-projective.

\smallskip

Let us collect the notations we will use. For more notions, see \cite{KM98} and \cite{Fuj17}.

\begin{itemize}
    \item For a normal variety $X$, and an effective $\Q$-Weil divisor $\Delta$ on $X$, we say that $(X,\Delta)$ is a \emph{couple}. The terminology is imported from \cite[2.8]{Bir19}. If $(X,\Delta)$ is a couple with $K_X+\Delta$ $\Q$-Cartier, then we say $(X,\Delta)$ is a \emph{pair}.
    \item Let $X$ be a variety. A morphism $f:X'\to X$ is a \emph{resolution} if it is a proper birational morphism with $X'$ smooth. For a couple $(X,\Delta)$, a resolution $f:X'\to X$ is a \emph{log resolution} if $\mathrm{Exc}(f)\cup \Supp f^{-1}_*\Delta$ is a simple normal crossing (snc) divisor.
    \item Let $S$ be a scheme, and let $s\in S$ be a point. We denote by $k(s)$ the residue field of $\mathcal{O}_{S,s}$. We can regard $s$ as a morphism $s:\Spec k(s)\to S$. We define $\overline{s}$ by the morphism $\Spec \overline{k(s)}\to S$. For an integral scheme $X$ we denote by $k(X)$ the function field of $X$.
    \item Let $f:X\to S$ be a morphism of schemes, and let $s\in S$ be a point. Define $X|_s:=X\times_S s$ and $X|_{\overline{s}}:=X\times_S \overline{s}$.
    \item Let $(X,\Delta)$ be a pair, $f:X'\to X$ a proper birational morphism between two normal varieties, and $E$ a prime divisor on $X'$. We define the \emph{log discrepancy} by
    $$ A_{X,\Delta}(E):=\mathrm{mult}_E(K_{X'}-f^*(K_X+\Delta))+1.$$
    Note that the definition does not depend on the choice of $f$. If $A_{X,\Delta}(E)>0$ for all prime divisors $E$ over $X$, then we say that $(X,\Delta)$ is \emph{kawamata log terminal (klt)}.
    \item Consider $X$ as a normal variety, and $\mathcal{F}$ as a coherent sheaf on $X$. We define
    $$ \mathcal{F}^{\vee}:=\mathcal{H}om_{\mathcal{O}_X}(\mathcal{F},\mathcal{O}_X),$$
    and refer to $\mathcal{F}^{\vee\vee}$ as the \emph{double dual} of $\mathcal{F}$. There exists a natural map $\mathcal{F}\to \mathcal{F}^{\vee\vee}$. We say that $\mathcal{F}$ is a \emph{reflexive sheaf} if this natural map is an isomorphism.
    \item Let $f:X\to S$ be a morphism of varieties, and $D$ a Cartier divisor on $X$. We define the base ideal $\mathfrak{b}(X/S,|D|)$ as the image of the natural map
    $$ f^*f_*\mathcal{O}_X(D)\otimes \mathcal{O}_X(-D)\to \mathcal{O}_X.$$
    Let us say $D$ is \emph{movable over $S$} if there is a positive integer $m$ such that the scheme in $X$ defined by $\mathfrak{b}(X/S,|mD|)$ has codimension $\ge 2$.
    \item Let $X$ be a normal variety, and let $\mathcal{F}$ a coherent sheaf on $X$. We say that $\mathcal{F}$ is \emph{globally generated in codimension $1$} if there is an open subscheme $U\subseteq X$ with $\mathrm{codim}_X(X\setminus U)\ge 2$ such that $\mathcal{F}|_U$ is a globally generated sheaf on $U$.
\end{itemize}

\subsection{Strongly $F$-regular couple}
In this subsection, we define the notion of \emph{strongly $F$-regular couple}. Let $X$ be a variety over a field of characteristic $p>0$. Then we can define the \emph{Frobenius} $F_X:X\to X$ as follows: If $X=\Spec R$ is affine, then the Frobenius is just the Frobenius on $R$. If $X$ is general, then we can patch an affine chart of $X$ to define $F_X:X\to X$.

\begin{definition}
Let $(X,\Delta)$ be an affine couple over a characteristic $p>0$ field. We say that $(X,\Delta)$ is \emph{strongly $F$-regular} if for every effective divisor $D$ on $X$, there exists some $e>0$ such that the natural map
$$ \mathcal{O}_X\hookrightarrow F^e_{X*}\mathcal{O}_X(\ceil{(p^e-1)\Delta}+D)$$
splits.
\end{definition}

Let us recall the main theorems in \cite{SS10}.

\begin{theorem}[{cf. \cite[Theorem 4.3]{SS10}}] \label{SS10}
Let $(X,\Delta)$ be a strongly $F$-regular couple. Then there is an effective $\Q$-Weil divisor $\Delta'$ on $X$ such that $(X,\Delta+\Delta')$ is a strongly $F$-regular pair.
\end{theorem}

\begin{proof}  
We can apply \cite[Theorem 4.3]{SS10} and there is an effective $\Q$-Weil divisor on $X$ such that $(X,\Delta+\Delta')$ is strongly $F$-regular and $K_X+\Delta+\Delta'$ is anti-ample.
\end{proof}

The notion of strong $F$-regularity is much stronger than the notion of klt in the case of a pair.

\begin{theorem}[{cf. \cite[Theorem 3.3]{HW02}}] \label{real}
Let $(X,\Delta)$ be a strongly $F$-regular pair. Then $(X,\Delta)$ is klt.
\end{theorem}

\subsection{Valuation}
In this subsection, we briefly recall basic notions from valuation theory. For background, see \cite{JM12}, \cite{Can20}, or \cite[§1.2]{Xu25}. We do not assume that $\operatorname{char} k=0$.

Let $X$ be a normal variety over a field $k$. A \emph{valuation} over the function field $k(X)$ is a map
$$
\nu:k(X)^{\times}\to \R
$$
such that
\begin{itemize}
    \item[(a)] $\nu(a)=0$ for all $a\in k^{\times}$,
    \item[(b)] $\nu(fg)=\nu(f)+\nu(g)$ for all $f,g\in k(X)^{\times}$, and
    \item[(c)] $\nu(f+g)\ge \min\{\nu(f),\nu(g)\}$ for all $f,g\in k(X)$.
\end{itemize}
We extend $\nu$ to $k(X)$ by setting $\nu(0):=+\infty$.
We say a valuation $\nu$ over $k(X)$ is \emph{centered} on $X$ if there is an affine open subscheme $U\subseteq X$ such that if we write $U=\Spec R$, then
$$ R\subseteq \mathcal{O}_{\nu}:=\{f\in k(X)\,|\,\nu(f)\ge 0\}.$$
Let $\mathfrak{m}_{\nu}:=\{f\in \mathcal{O}_{\nu}\,|\,\nu(f)>0\}$ be the maximal ideal in $\mathcal{O}_{\nu}$. We denote the point given by the prime ideal $R\cap \mathfrak{m}_{\nu}$ to be the \emph{center} $c_X(\nu)\subseteq X$. If a valuation $\nu$ over $k(X)$ is centered on $X$, then we say that $\nu$ is a valuation \emph{over} $X$. Let $x\in X$ be a (not necessarily closed) point, and define
$$ \mathrm{Val}_X:=\{\nu\text{ is a valuation over }X\},$$
$$ \mathrm{Val}_{X,x}:=\{\nu\text{ is a valuation over }X\text{ with }x=c_X(\nu)\},$$
and $\mathrm{Val}^*_X:=\mathrm{Val}_X\setminus \{0\}$.

\smallskip

If $E$ is a prime divisor over $X$, then the order of vanishing
$$
\operatorname{ord}_E:k(X)^{\times}\to \R
$$
is a valuation over $X$. For any $c>0$, a valuation of the form $c\cdot \operatorname{ord}_E$ is called \emph{divisorial}.

\smallskip

Let $f:X'\to X$ be a log resolution, and let $E_1,\dots,E_r$ be the irreducible components of a reduced simple normal crossings divisor
$$
E:=\sum_{i=1}^r E_i \subset X'.
$$
If $f$ is an isomorphism over $X\setminus \bigcup_i f(E_i)$, then $(X',E)\to X$ is called a \emph{log smooth model}. Fix $\alpha=(\alpha_1,\cdots,\alpha_r)\in \R_{\ge0}^r$ and assume $\bigcap_{i=1}^r E_i\neq\varnothing$. Let $C$ be a connected component of $\bigcap_i E_i$ with generic point $\eta$. Choose regular parameters $z_1,\dots,z_r\in \mathcal{O}_{X',\eta}$ such that $E_i=(z_i=0)$ for all $i$. For $f\in \mathcal{O}_{X',\eta}$ write
$$
f=\sum_{\beta\in \Z_{\ge0}^r} c_{\beta}\, z^{\beta},\,\,
z^{\beta}:=z_1^{\beta_1}\cdots z_r^{\beta_r}.
$$
Define
$$
\nu_{(X',E),\alpha}(f):=\min\left\{\sum_{i=1}^r \alpha_i \beta_i \ \Bigm|\ c_\beta\neq 0\right\}.
$$
This gives a valuation on $k(X)$, independent of the choice of such parameters. All such valuations are called \emph{quasi-monomial}. We denote by $\mathrm{QM}(X',E)$ the set of quasi-monomial valuations defined by $(X',E)$, and by $\mathrm{QM}_{\eta}(X',E)\subseteq \mathrm{QM}(X',E)$ those whose center equals $\eta$.

\smallskip

For a valuation $\nu$ over $X$ and an ideal sheaf $\mathfrak{a}\subset \mathcal{O}_X$, set
$$
\nu(\mathfrak{a}):=\min\{\nu(f)\mid f\in \mathfrak{a}_x\},\,\, x:=c_X(\nu).
$$
If $\mathfrak{a},\mathfrak{b}$ are coherent ideals, then
$$
\nu(\mathfrak{a}\mathfrak{b})=\nu(\mathfrak{a})+\nu(\mathfrak{b}).
$$

\smallskip

For any valuation $\nu$ over $X$ and for any ideal sheaf $\mathfrak{a}$ on $X$, we define
$$ \nu(\mathfrak{a}):=\min\{\nu(f)\,|\,f\in \mathfrak{a}_x\text{ where }x=c_X(\nu)\}.$$
Note that for any two ideals $\mathfrak{a},\mathfrak{b}$,
$$ \nu(\mathfrak{a}\mathfrak{b})=\nu(\mathfrak{a})+\nu(\mathfrak{b})$$

\smallskip

We can give a topology on $\mathrm{Val}_X$ as follows: the topology of $\mathrm{Val}_X$ is the weakest topology such that $\nu\mapsto \nu(\mathfrak{a})$ is a continuous function for any valuation $\nu$ over $X$.

\subsection{Log discrepancy}
In this subsection, we define the \emph{log discrepancy function} for any valuation. For a log smooth model $\left(X',E:=\sum^r_{i=1}E_i\right)\to X$ and a valuation $\nu$ over $X$, we define $\rho_{(X',E)}:\mathrm{Val}_X\to \mathrm{QM}(X',E)$ as
$$ \rho_{(X',E)}(\nu):=\nu_{(X',E),\left(\nu(E_1),\cdots,\nu(E_r)\right)}.$$

\begin{definition}
Let $(X,\Delta)$ be a klt pair, the \emph{log discrepancy function}
$$ A_{X,\Delta}:\mathrm{Val}_X\to [0,+\infty]$$
is defined as follows:
\begin{itemize}
    \item Let $\left(X',E:=\sum^r_{i=1}E_i\right)\to X$ be a log-smooth model of $X$, and let $\eta$ be the generic point of an irreducible component of $\bigcap E_i$. Let $\alpha:=(\alpha_1,\cdots,\alpha_r)\in \R^r_{\ge 0}$ be a tuple. We define
    $$ A_{X,\Delta}(\nu_{(X',E),\eta,\alpha}):=\sum^r_{i=1}\alpha_i\cdot A_{X,\Delta}(E_i).$$
    Note that if $\left(X'',E':=\sum^{r'}_{i=1} E'_i\right)\to X$ is another log smooth model such that there is $k$ prime components $E'_1,\cdots,E'_k$ with $\eta$ a generic point of $\bigcap^k_{i=1} E'_i$, $\alpha'\in \R^{r'}_{\ge 0}$, and $\nu_{(X',E),\eta,\alpha}=\nu_{(X'',E'),\eta,\alpha'}$, then $A_{X,\Delta}(\nu_{(X',E),\eta,\alpha})=A_{X,\Delta}(\nu_{(X'',E'),\eta,\alpha'})$ by \cite[Proposition 5.1]{JM12} and \cite[Lemma 4.3]{Can20}.
    \item Let us assume that the base field of $X$ is characteristic $0$, and let $\nu$ be a general valuation over $X$. We define
    $$ A_{X,\Delta}(\nu):=\sup_{(Y,E)\text{ log smooth model}}A_{X,\Delta}(\rho_{(Y,E)}(\nu)).$$
\end{itemize}
\end{definition}

\subsection{Log canonical threshold and Jonsson-Mustață conjecture}

Let $X$ be a normal variety, $\Phi\subseteq \Z_{\ge 0}$ a semigroup, and let $\mathfrak{a}_{\bullet}:=\{\mathfrak{a}_m\}_{m\in \Phi}$ be a set of ideals on $X$. We say $\mathfrak{a}_{\bullet}$ is a \emph{graded sequence of ideals} if
$$ \mathfrak{a}_m\cdot \mathfrak{a}_m\subseteq  \mathfrak{a}_{m+n}\text{ for all }m,n\in \Phi.$$

\begin{definition}
Let $(X,\Delta)$ be a klt pair, let $\mathfrak{a}$ be an ideal in $\mathcal{O}_X$ and let $\mathfrak{a}_{\bullet}:=\left\{\mathfrak{a}_m\right\}_{m\in \Phi}$ be a graded sequence of ideals in $\mathcal{O}_X$. We define the following:
$$ \nu(\mathfrak{a}_{\bullet}):=\inf_{m\in \Phi}\frac{\nu(\mathfrak{a}_m)}{m}\text{ for any }\nu\in \mathrm{Val}_X,$$
$$ \mathrm{lct}(X,\Delta,\mathfrak{a}):=\inf_{\nu\in \mathrm{Val}^*_X}\frac{A_{X,\Delta}(\nu)}{\nu(\mathfrak{a})},$$
and
$$ \mathrm{lct}(X,\Delta,\mathfrak{a}_{\bullet}):=\inf_{\nu\in \mathrm{Val}^*_X}\frac{A_{X,\Delta}(\nu)}{\nu(\mathfrak{a}_{\bullet})}.$$
We say that a valuation $\nu$ over $X$ \emph{computes} $\mathrm{lct}(X,\Delta,\mathfrak{a}_{\bullet})$ if
$$ \mathrm{lct}(X,\Delta,\mathfrak{a}_{\bullet})=\frac{A_{X,\Delta}(\nu)}{\nu(\mathfrak{a}_{\bullet})}.$$
\end{definition}

The following is an important theorem regarding the log canonical threshold of a graded sequence of ideals. It states that the infimum defining the log canonical threshold is achieved by a specific type of valuation. We refer to this theorem as the \emph{Jonsson-Mustață conjecture}, as it follows from \cite[Conjecture 7.4]{JM12} by setting $\mathfrak{q}=0$.

\begin{theorem}[{cf. \cite[Conjecture 7.4]{JM12} and \cite[Theorem 1.1]{Xu20}}] \label{quasar}
Let $(X,\Delta)$ be a klt pair, and let $\mathfrak{a}_{\bullet}$ be a graded sequence of ideals on $X$. Then there is a quasi-monomial valuation $\nu$ that computes $\mathrm{lct}(X,\Delta,\mathfrak{a}_{\bullet})$.
\end{theorem}

Let us recall the following lemma (cf. \cite[Lemma 6.1 and Corollary 6.4]{JM12}  and \cite[Lemma 2.7]{Kim25b}).

\begin{lemma}[{cf. \cite[Lemma 6.1 and Corollary 6.4]{JM12} and \cite[Lemma 2.7]{Kim25b}}] \label{youareeasy'}
Let $X$ be a variety over $k$, and let $\mathfrak{a}_{\bullet}$ be a graded sequence of ideals in $\mathcal{O}_X$. Then the function $\nu\mapsto \nu(\mathfrak{a}_{\bullet})$ is upper-semicontinuous on $\mathrm{Val}_X$. Moreover, if $k$ is characteristic $0$, then the function is continuous on $\mathrm{Val}_X\cap \{A_{X,\Delta}(\nu)<\infty\}$.
\end{lemma}

\subsection{Models}
In this subsection, we define the notion of \emph{model} that is necessary to allow us to use the reduction mod $p$ method. For more information about the notion of model, see \cite[Chapter 2]{HH99}.

\begin{definition}
Let $X$ be a variety over $\C$, let $D$ be an effective Weil divisor on $X$, and $f:X'\to X$ a morphism between varieties. Suppose $R$ is a finitely generated $\Z$-algebra.
\begin{itemize}
    \item[(a)] If $X_R$ a flat $R$-scheme with $X_R\times_R \C=X$, then we say that $X_R$ is a \emph{model} of $X$ over $R$.
    \item[(b)] If $X_R$ is a model of $X$ over $R$, and if $D_R\subseteq X_R$ is an effective Weil divisor on $X_R$ flat over $R$ with $(D_R\subseteq X_R)\times_R \C=D\subseteq X$, then we say that $D_R$ is a \emph{model} of $X$ over $R$.
    \item[(c)] If $f_R:X'_R\to X_R$ is an $R$-morphism with $f_R\times_R \C=f$ and $X_R,X'_R$ flat over $R$, then we say that $f_R$ is a \emph{model} of $f$ over $R$.
    \item[(d)] Let $E$ be an effective divisor on $E$ such that $(X',E)\to X$ is a log smooth model. If $f_R:X'_R\to X_R$ and $E_R$ are models of $f$ and $E$ over $R$ respectively, $E_R$ is snc, and $f_R$ is an isomorphism outside $E_R$, then we say that $(X'_R,E_R)\to X_R$ is a \emph{model} of $(X',E)\to X$ over $R$.
    \item[(e)] If $X_R$ is a model of $X$ over $R$,
    $$ D=\sum q_i D_i$$
    is a $\Q$-Weil divisor with $q_i\in \Q$ and prime $D_i$, and $D_{iR}$ is a model of $D_i$ over $R$ for each $i$, then we say that
    $$ D_R:=\sum q_iD_{iR}$$
    is a \emph{model} of $D$ over $R$. Note that if $D$ is $\Q$-Cartier, then we may choose $D_R$ as a $\Q$-Cartier divisor on $X_R$.
    \item[(f)] Let $(X,\Delta)$ be a couple over $\C$, $X_R$ a model of $X$ over $R$, and $\Delta_R$ is so. Then we say that $(X_R,\Delta_R)$ is a \emph{model} of $(X,\Delta)$ over $R$.
\end{itemize}
\end{definition}

\begin{notation} \label{nota}
Fix a finitely generated $\Z$-algebra $R$ with closed point $s$. Let $X_R,X'_R$ be flat $R$-schemes and $f_R:X'_R\to X_R$ be an $R$-morphism. We denote by
$$ X_{\overline{s}}:=X_R|_{\overline{s}}, X'_{\overline{s}}:=X'_R|_{\overline{s}}, f_{\overline{s}}:=f_R|_{\overline{s}}.$$
If
$$ D_R=\sum r_i D_i$$
is a Weil $\R$-divisor on $X_R$ such that each $D_i$ is prime and flat over $R$, then we denote
$$ D_{\overline{s}}:=\sum r_i D_{R}|_{\overline{s}} $$
as a Weil $\R$-divisor on $X_{\overline{s}}$. Note that if $D_R$ is $\Q$-Cartier, then $D_{\overline{s}}$ is $\Q$-Cartier as well. 

\smallskip

Let $\left(X',E:=\sum^r_{i=1}E_i\right)\to X$ be a log smooth model, and $\left(X'_R,E_R:=\sum^r_{i=1}E_{iR}\right)\to X_R$  a model of $(X',E)\to X$ over $R$. Then for every $1\le i_1<\cdots<i_k\le r$ and $s\in \Spec R$, there is a one-to-one correspondence between connected components of $\bigcap^k_{i=1}E_{i_j}$ and $\bigcap^k_{i=1}E_{i_j\overline{s}}$, and if $\eta$ is a generic point of $\bigcap^k_{i=1}E_{i_j}$, then we define $\eta_{\overline{s}}$ as the generic point of $\bigcap^k_{i=1}E_{i_j\overline{s}}$ corresponding to $\eta$.
\end{notation}

Let us prove a proposition about log discrepancy in a family

\begin{proposition} \label{disc 2}
Let $(X,\Delta)$ be a pair over $\C$, $f:X'\to X$ a proper birational morphism with $E$ a prime divisor on $X'$, $R$ a finitely generated $\Z$-algebra and $(X_R,\Delta_R)$ and $f_R:X'_R\to X_R$ models of $(X,\Delta)$ and $f$ over $R$ respectively. Denote by $E_R$ the model of $E$ over $R$ defined by $f_R$. Then,
$$ A_{X,\Delta}(E)=A_{X_{\overline{s}},\Delta_{\overline{s}}}(E_{\overline{s}})$$
for an $s\in \Spec R$ (for the notations, see Notation \ref{nota}).
\end{proposition}

\begin{proof}
Observe
$$
\begin{aligned}
\mathrm{mult}_{E_R}(f^*_R(K_{X_R}+\Delta_R)-K_{X'_R})+1&=\mathrm{mult}_{E}(f^*(K_{X}+\Delta)-K_{X'})+1
\\ &=A_{X,\Delta}(E),
\end{aligned}$$
and similarly, 
$$ \mathrm{mult}_{E_R}(f^*_R(K_{X_R}+\Delta_R)-K_{X'_R})+1=A_{X_{\overline{s}},\Delta_{\overline{s}}}(E_{\overline{s}})$$
for an $s\in \Spec R$.
\end{proof}

We define the notion of \emph{strongly $F$-regular type}.

\begin{definition} \label{str}
Let $(X,\Delta)$ be a couple over $\C$, and let $(X_R,\Delta_R)$ be a model of $(X,\Delta)$ over $R$. We say $(X,\Delta)$ is a \emph{strongly $F$-regular type couple} if for a general $s\in \Spec R$, $(X_{\overline{s}},\Delta_{\overline{s}})$ is a strongly $F$-regular couple (Notation \ref{nota}).
\end{definition}

\subsection{Toroidal blowup} \label{toroidal}
In this subsection, we define \emph{toroidal blowup} and prove a lemma.

\smallskip

Let $\left(X',E=\sum^r_{i=1}E_i\right)$ be a couple of a regular scheme $X'$ and an snc divisor $E\subseteq X'$. Let $1\le i_1<i_2<\cdots<i_k\le r$. We say that the blowup of $X'$ along $\bigcap^k_{j=1}E_{i_j}$ is a \emph{toroidal blowup}. Note that any toroidal blowup is a regular scheme because it is a blowup of a regular scheme along a closed regular subscheme.

\smallskip

Let us prove the following lemma:

\begin{lemma} \label{detail}
Let $X$ be a normal variety over $\C$, and let $\left(X',E=\sum^r_{i=1}E_i\right)$ be a log-smooth model over $X$, and $\eta$ the generic point of a connected component of the intersection of $k$ prime components of $E$. Then there exists a finitely generated $\Z$-algebra $R$ and models $X_R$ and $\left(X'_R,E_R=\sum^r_{i=1}E_{iR}\right)$ of $X$, $(X',E)$ over $R$ respectively with the following property: 

\smallskip

For every $\alpha\in \Q^k_{\ge 0}$, there is the composition of a sequence of toroidal blowups $f':X''_R\to X'_R$, a prime divisor $F_R$ on $X''_R$ and a positive number $c>0$ such that 
$$\nu_{(X',E),\eta,\alpha}=c\cdot \mathrm{ord}_F\,\,\,\text{ and }\,\,\,\nu_{(X'_{\overline{s}},E_{\overline{s}}),\eta_{\overline{s}},\alpha}=c\cdot \mathrm{ord}_{F_{\overline{s}}}$$
for every closed point $s\in \Spec R$ (Notation \ref{nota}).
\end{lemma}

\begin{proof}
There is a finitely generated $\Z$-algebra $R$, and models $X_R$ and $(X'_R,E_R)$ of $X$ and $(X',E)$ over $R$ respectively.

\smallskip

Note that for every $\alpha:=(\alpha_1,\cdots,\alpha_r)\in \Q^r_{>0}$, there is the composition of a sequence of toroidal blowups $f':X''\to X'$ and a prime component $F$ of $\mathrm{Exc}(f')\cup \Supp (f')^{-1}_*E$ such that there is a positive number $c>0$ such that
$$ \nu_{(X',E),\eta,\alpha}=c\cdot \mathrm{ord}_F$$
(cf. \cite[Lemma 3.6 (ii)]{JM12}). Moreover, since $f'$ is the composition of a sequence of toroidal blowups, we obtain that there is a model $f'_R:X''_R\to X'_R$ of $f'$ over $R$. Note that $\nu_{(X'_{\overline{s}},E_{\overline{s}}),\eta_{\overline{s}},\alpha}$ can be characterized as follows: $\nu_{(X'_{\overline{s}},E_{\overline{s}}),\eta_{\overline{s}},\alpha}$ is the divisorial valuation in $\mathrm{QM}_{\eta}(X'_{\overline{s}},E_{\overline{s}})$ such that
$$ \nu_{(X'_{\overline{s}},E_{\overline{s}}),\eta_{\overline{s}},\alpha}(E_{i\overline{s}})=\alpha_i.$$

\smallskip

Since $f'_{\overline{s}}:X''_{\overline{s}}\to X'_{\overline{s}}$ is the composition of a sequence of toroidal blowups of $(X'_{\overline{s}},E_{\overline{s}})$, $c\cdot \mathrm{ord}_{F_{\overline{s}}}\in \mathrm{QM}_{\eta}(X'_{\overline{s}},E_{\overline{s}})$. Moreover, since
$$ c\cdot \mathrm{mult}_F((f')^*E_i)=c\cdot \mathrm{ord}_F(E_i)=\alpha_i,$$
we obtain that $c\cdot \mathrm{mult}_{F_R}((f'_R)^*E_{iR})=\alpha_i$, and therefore 
$$c\cdot \mathrm{mult}_{F_{\overline{s}}}((f'_{\overline{s}})^*E_{i\overline{s}})=c\cdot \mathrm{ord}_{F_{\overline{s}}}(E_{i\overline{s}})=\alpha_i.$$
Hence,
$$ \nu_{(X'_{\overline{s}},E_{\overline{s}}),\eta_{\overline{s}},\alpha}=c\cdot \mathrm{ord}_{F_{\overline{s}}}$$
as we want.
\end{proof}

\section{Natural valuation and natural pullback} \label{333}
In this section, we introduce the notion of \emph{natural valuation} and \emph{natural pullback}, and collect basic lemmas about them. We keep the standing notation of Notation \ref{nota}.

\begin{definition}
Let $X$ be a normal variety, and let $\nu\in \mathrm{Val}_X$ be a valuation over $X$. Suppose $D$ is a Weil divisor on $X$. We define
$$ \nu^{\natural}(D):=\nu^{\natural}(\mathcal{O}_X(-D)).$$
Moreover, we define
$$ \nu(D):=\lim_{m\to \infty}\frac{1}{m!}\nu^{\natural}(m!D).$$
Note that the limit exists by \cite[Lemma 2.8]{dFH09}. If $D$ is a $\Q$-Weil divisor, then we define
$$ \nu(D):=\frac{1}{m}\nu(mD),$$
where $m$ is a positive integer with $mD$ Weil.
\end{definition}

\begin{remark} \label{natural0}
Let $D$ be a Cartier divisor. Then $\nu^{\natural}(D)$ coincides with the notion of order of vanishing of $D$ along $\nu$.
\smallskip

For any two Weil divisors $D,D'$ on $X$, we have 
$$\nu^{\natural}(D+D')\le \nu^{\natural}(D)+\nu^{\natural}(D').$$
Indeed, let $U\subseteq X$ be an open subscheme. If $\phi\in \mathcal{O}_X(-D)(U)$ and $\phi'\in \mathcal{O}_X(-D')(U)$, then $\mathrm{div} \,\phi\ge D$ and $\mathrm{div}\,\phi'\ge D'$ on $U$, and therefore $\mathrm{div}\,(\phi\cdot \phi')\ge D+D'$. Thus, $\phi \cdot \phi'\in \mathcal{O}_X(-D-D')(U)$, and $\mathcal{O}_{X}(-D)\cdot \mathcal{O}_X(-D')\subseteq \mathcal{O}_X(-(D+D'))$. The equality holds if either $D$ or $D'$ is Cartier (cf. \cite[Proposition 2.10]{dFH09}).

\smallskip

For any two $\Q$-Weil divisors $D,D'$ on $X$,
$$ \nu(D+D')\le \nu(D)+\nu(D'),$$
and the equality holds if either $D$ or $D'$ is $\Q$-Cartier. Unfortunately, the equality does not hold in general. See \cite[Example 2.12]{dFH09}.
\end{remark}

We define the notion of \emph{natural pullback}, which serves as a version of pullback in the case of a Weil divisor.

\begin{definition}
Let $f:X'\to X$ be a proper birational morphism between two normal varieties, and let $D$ be a Weil divisor on $X$. We define $f^{\natural}D$ as the Weil divisor on $X'$ with
$$ \mathcal{O}_{X'}(-f^{\natural}D)=(\mathcal{O}_X(-D)\cdot \mathcal{O}_{X'})^{\vee\vee}.$$
Moreover, we define
$$ f^*D:=\lim_{m\to \infty}\frac{1}{m!}f^*(m!D)$$
as an $\R$-Weil divisor on $X'$.
\end{definition}

\begin{remark} \label{natural}
Let us collect basic facts about the natural pullback.

\smallskip

\begin{itemize}

\item[(a)] From the natural map $f^*\mathcal{O}_X(-D)\to (f^*\mathcal{O}_X(-D))^{\vee\vee}$ and through adjunction, there is an induced natural map $\mathcal{O}_X(-D)\to f_*(f^*\mathcal{O}_X(-D))^{\vee\vee}$. Additionally, via the multiplication map 
$$f^*\mathcal{O}_X(-D):=\mathcal{O}_X(-D)\otimes_{\mathcal{O}_X}\mathcal{O}_{X'}\to \mathcal{O}_X(-D)\cdot \mathcal{O}_{X'}$$
inside $k(X')$, another natural map arises: $\mathcal{O}_X(-D)\to f_*\mathcal{O}_{X'}(-f^{\natural}D)$.

\item[(b)] Given two linearly equivalent Weil divisors $D$ and $D'$ on $X$, there exists an isomorphism $\imath:\mathcal{O}_X(-D)\cong \mathcal{O}_X(-D')$. Consequently, we have an isomorphism
$$ \imath\cdot \mathcal{O}_{X'}:\mathcal{O}_X(-D)\cdot \mathcal{O}_{X'}\cong \mathcal{O}_X(-D')\cdot \mathcal{O}_{X'}.$$
By taking the double dual on both sides, we conclude that $f^{\natural}D\sim f^{\natural}D'$, demonstrating that the natural pullback preserves linear equivalence.

\item[(c)] We can write $f^{\natural}D$ as follows:
$$ f^{\natural}D=\sum_{E\text{ prime ideal on }X'}\mathrm{ord}^{\natural}_E(D)E.$$
Moreover, if $E$ is a prime divisor on $X'$, then $\mathrm{ord}^{\natural}_E(D)=\mathrm{ord}^{\natural}_E(f^{\natural}D)$. In fact,
$$ 
\begin{aligned}
    \mathrm{ord}^{\natural}_E(D)&=\mathrm{ord}_E(\mathcal{O}_X(-D))
    \\ &=\mathrm{ord}_E(\mathcal{O}_X(-D)\cdot \mathcal{O}_{X'})
    \\ &=\mathrm{ord}_E((\mathcal{O}_X(-D)\cdot \mathcal{O}_{X'})^{\vee\vee})
    \\ &=\mathrm{ord}_E(\mathcal{O}_X(-f^{\natural}D))
    \\ &=\mathrm{ord}^{\natural}_E(f^{\natural}D).
\end{aligned}
$$
Note that if $E$ is just a prime divisor over $X'$, not on $X'$, then the statement is false. See \cite[Example 2.12]{dFH09}.

\item[(d)] Unfortunately, because of the double dual, natural pullback lacks many features of functoriality of $f$. For example, we do not have $(f\circ g)^{\natural}D=g^{\natural}f^{\natural}D$, where $g:X''\to X'$ is a proper birational morphism between normal varieties. Instead of that, we have $(f\circ g)^{\natural}D-g^{\natural}f^{\natural}D$ is effective $g$-exceptional. Moreover, if $\mathcal{O}_X(-D)\cdot \mathcal{O}_{X'}$ is a line bundle, then $(f\circ g)^{\natural}D=g^{\natural}f^{\natural}D$ (cf. \cite[Remark 2.13]{dFH09} and \cite[Remark 2.12]{CEMS18}).

\item[(e)] Let $f:X'\to X$ be a small morphism. Then $f^{\natural}D$ is the strict transform of $D$, and thus
$$ f^{-1}_*D=f^{\natural}D=\frac{1}{m!}f^{\natural}(m!D)=f^{*}D$$
for every positive integer $m$ (cf. \cite[Remark 2.12]{CEMS18}).

\end{itemize}
\end{remark}

\begin{lemma} \label{q-birational}
Let $f:X'\to X$ be a proper birational morphism between two normal varieties, and let $D$ be a Weil divisor on $X$. Then the natural map $\mathcal{O}_X(-D)\to f_*\mathcal{O}_{X'}(-f^{\natural}D)$ constructed in Remark \ref{natural} (a) is an isomorphism.
\end{lemma}

\begin{proof}
The proof of the lemma is the same as \cite[Lemma 3.6]{Kim25a}, and we recall the proof for the convenience of the readers.

\smallskip

Let $\mathcal{Q}$ be the cokernel of the natural morphism. Then we have an exact sequence
$$0\to \mathcal{O}_X(-D)\to f_*\mathcal{O}_{X'}(-f^{\natural}D)\to \mathcal{Q}\to 0.$$
Let $x$ be a codimension $\ge 2$ point in $X$. Note that $H^1_x(X,\mathcal{O}_X(-D))=0$ since $\mathcal{O}_X(-D)$ is a reflexive sheaf on $X$. By taking local cohomology, we obtain $ H^0_x(X,\mathcal{Q})=0$. Note that $\codim \Supp\mathcal{Q}\ge 2$. Hence, $\mathcal{Q}=0$.
\end{proof}

\begin{lemma} \label{pullbackpreservessurjection}
Let $X$ be an affine normal variety, and let $A$ be a Weil divisor on $X$. Then $\mathcal{O}_{X'}(-f^{\natural}(-A))$ is globally generated in codimension $1$.
\end{lemma}

\begin{proof}
Note that $\mathcal{O}_X(A)$ is globally generated. By definition of global generation, there are $s_1,\cdots,s_r\in H^0(X,\mathcal{O}_X(A))$ that induce a surjection
$$ \bigoplus \mathcal{O}_X\to \mathcal{O}_X(A),\,(f_1,\cdots,f_r)\mapsto \sum s_if_i.$$
Taking $f^*$ and considering the multiplication map $f^*\mathcal{O}_X(A)\to \mathcal{O}_X(A)\cdot \mathcal{O}_{X'}$ gives us a surjection
$$ \bigoplus \mathcal{O}_{X'}\to \mathcal{O}_X(A)\cdot \mathcal{O}_{X'},\,(f_1,\cdots,f_r)\mapsto \sum s_if_i.$$
Moreover, since $X'$ is normal, there is an open subscheme $U\subseteq X'$ with $\mathrm{codim}_{X'}(X'\setminus U)\ge 2$,
and in fact, $\mathcal{O}_X(A)\cdot \mathcal{O}_{X'}|_U$ is globally generated, and thus its double dual, $\mathcal{O}_{X'}(-f^{\natural}(-A))$ is globally generated in codimension $1$.
\end{proof}

\begin{remark} \label{nothope}
In general, $\mathcal{O}_{X'}(-f^{\natural}(-A))$ is not globally generated. For example, let $X:=(xy=zw)\subseteq \A^4$, and let $A:=(x=z=0)\subseteq X$. Then the blowup of $X$ along $\mathcal{O}_X(-A)$ gives us a small resolution $f:X'\to X$, and $\mathcal{O}_{X'}(-f^{\natural}(-A))|_E\cong \mathcal{O}_{\P^1}(-1)$ is not globally generated.

\smallskip

Moreover, if we replace $\mathcal{O}_{X'}(-f^{\natural}(-A))$ with $\mathcal{O}_{X'}(f^{\natural}A)$, then in general, Lemma \ref{pullbackpreservessurjection} does not hold. Indeed, let $B:=(x=w=0)\subseteq X$, let $X''$ be the blowup of $X$ along $\mathcal{O}_X(-B)$, and let $X'\dashrightarrow X''$ be the Atiyah flop. Let $g:X'''\to X$ be the blowup of $X$ at the origin, and let $E$ be the exceptional divisor on $X'''$. Define by $H\subseteq X'$ a general hyperplane section, and $D$ the strict transform of $H$ to $X$. Then, as computed in \cite[Example 2.12]{dFH09}, we obtain $\mathrm{ord}^{\natural}_E(D)>0$, whereas $\mathrm{ord}^{\natural}_E(-D)=0$. Therefore, $\mathcal{O}_{X'''}(g^{\natural}D)$ is not globally generated in codimension $1$.
\end{remark}

We can think that the following lemma is a kind of tie-breaking argument for Weil divisors.

\begin{lemma} \label{zzm}
Let $X$ be an affine normal variety, $f:X'\to X$ a proper birational morphism, $E$ a prime divisor on $X$, and $A$ a Weil divisor on $X$. Then for a general effective divisor $D\in H^0(X,\mathcal{O}_X(A))$, we obtain $\mathrm{ord}^{\natural}_E(-D)=0$.
\end{lemma}

\begin{proof}
Since $\mathcal{O}_{X'}(-f^{\natural}(-A))$ is globally generated in codimension $1$ by Lemma \ref{pullbackpreservessurjection}, we obtain that for a general effective divisor $D'\in H^0(X',\mathcal{O}_{X'}(-f^{\natural}(-A)))$ such that 
$$\mathrm{ord}^{\natural}_{E}(-D')=0.$$
We define $D:=f_*D'$.

\smallskip

Let us assert that $D'=-f^{\natural}(-D)$. In fact, the map that comes from Remark \ref{natural} (a)
$$ H^0(X,\mathcal{O}_X(A))\to  H^0(X',\mathcal{O}_{X'}(-f^{\natural}(-A)))$$
is an isomorphism by (a) of Lemma \ref{q-birational}, the map is $D''\mapsto -f^{\natural}(-D'')$, and the inverse is $D'''\mapsto f_*D'''$ for any $D''\in H^0(X,\mathcal{O}_X(A))$ and any $D'''\in H^0(X',\mathcal{O}_{X'}(-f^{\natural}(-A)))$ respectively. Thus, the claim is established.

\smallskip

Consequently, referring to Remark \ref{natural} (c), we have 
$$ \mathrm{ord}_{E}(-D)=0,$$
and this is the desired result.
\end{proof}

\begin{remark}
By the reason of Remark \ref{nothope}, we only get $\ord^{\natural}_E(D)\ge 0$, and we cannot hope $\ord^{\natural}_E(D)=0$.
\end{remark}

Let us prove the following version of the Negativity lemma. Roughly, the lemma is a combination of \cite[Proposition 4.1]{BLX22} and \cite[Lemma 14]{dFDTT15}.

\begin{lemma} \label{듸듸<3<3}
Let $f:X'\to X$ be a resolution of a normal variety $X$ over $\C$, $E:=\sum^r_{i=1} E_i$ an snc divisor on $X'$, and assume that the exceptional locus of $f$ has pure codimension $1$. Then there is a finitely generated $\Z$-algebra $R$, models $f_R:X'_R\to X_R$ and $(X'_R,E_R)$ of $f$, $(X',E)$ respectively such that for each $s\in \Spec R$,
\begin{itemize}
    \item[\emph{(a)}] for the composition of every sequence of toroidal blowups $f':X''\to X'$ of $(X',E)$, there are curves $C_{1,X''},\cdots,C_{n_{X''},X''}$ in $X''$ movable in codimension $1$ and the flat $R_s$-schemes $C_{1,X''_{R_s}},\cdots,C_{n_{X''},X''_{R_s}}$
    with
    $$ C_{j,X''_{R_s}}\times_{R_s} \C=C_{j,X''}$$
    such that the reductions $C_{1,X''_{\overline{s}}},\cdots,C_{n_{X''},X''_{\overline{s}}}$ are movable curves in codimension $1$ on $X''_{\overline{s}}$.
    \item[\emph{(b)}] for every divisor $D$ on $X''$, if $D\cdot C_{j,X''}\ge 0$ for all $j$ and $(f\circ f')_*D\le 0$, then $D\le 0$.
\end{itemize}
Here, $R_s$ is the localization of the ring $R$ at the maximal ideal of $R$ corresponding to the closed point $s\in \Spec R$.
\end{lemma}

\begin{proof}
Since the lemma is local, we may assume $X$ is quasi-projective. Let $n\ge 3$ be the dimension of $X$.

\smallskip

Let us construct $C_{j,X''}$ and $C_{j,X''_{R_s}}$. There are models $f_R:X'_R\to X_R$ and $(X'_R,E_R)$ of $f$ and $(X',R)$ over a finitely generated $\Z$-algebra $R$ respectively. We may assume $R$ to be regular. Since $f':X''\to X$ is the composition of a sequence of toroidal blowups, we obtain a model $f'_R:X''_R\to X'_R$ of $f'$. Let $L_3,\cdots,L_n$ be very general hyperplane sections in $X$. Then we may assume that 
$$\left(f^{-1}(L_3\cap \cdots \cap L_n),E|_{f^{-1}(L_3\cap \cdots \cap L_n)}\right)$$
is an snc pair, and therefore after further shrinking $R$, we may assume that there are models $L_{3R},\cdots,L_{nR}$ of $L_3,\cdots,L_n$ over $R$ such that 
$$Y'_e:=(f_R\circ f'_R)^{-1}(L_{(e+1)R}\cap \cdots \cap L_{nR})$$
is a regular scheme for each $e\ge 2$ (this is possible since 
\begin{equation} \label{snc}
Y_e:=\left(f^{-1}_R(L_{(e+1)R}\cap \cdots \cap L_{nR}),E|_{f^{-1}_R(L_{(e+1)R}\cap \cdots\cap L_{nR})}\right)\text{ is an snc pair},\end{equation}
and $Y'_e$ is obtained by the composition of a sequence of toroidal blowups of $Y_e$). Let $H_{3R_s},\cdots,H_{nR_s}$ be hyperplane sections in $X''_{R_s}:=X''_R\times_R R_s$ such that 
$$T_{eR_s}:=(H_{3R_s}\cap\cdots\cap H_{eR_s})\cap Y'_e$$
is a regular scheme of dimension $\dim R_s+2$ for every $3\le e\le n$ (cf. \cite[Theorem 2.17]{BMP+20}).

\smallskip

Set
$$ T_e:=T_{eR_s}\times_{R_s} \C,$$
and set $S_{e{R_s}}$ to be the normalization of $(f_{R_s}\circ f'_{R_s})(T_{eR_s})$. Let $C_{e,1,X''_{R_s}},\cdots,C_{e,m_e,X''_{R_s}}$ be the prime exceptional loci of $T_{eR_s}\to S_{eR_s}$ with 
$$C_{e,j,X''}:=C_{e,j,X''_{R_s}}\times_{R_s} \C$$ 
a curve on $X''$, and let 
$$C_{e,1,X''_{\overline{s}}}:=C_{e,1,X''_{R_s}}\times_{R_s} \overline{s},\cdots,C_{e,m_e,X''_{\overline{s}}}=C_{e,m_e,X''_{R_s}}\times_{R_s} \overline{s}.$$
Note that because of \cite[Lemma 02FZ]{Stacks} and the construction of $C_{e,j,X''_{R_s}}$, $\dim C_{e,j,X''_{\overline{s}}}\ge 1$. Moreover, $\dim C_{e,j,X''_{\overline{s}}}\le 1$ since $C_{e,j,X''_{\overline{s}}}$ is properly contained in one of the irreducible components of the surface $T_{e\overline{s}}$. Hence, we obtain that $C_{e,j,X''_{\overline{s}}}$ are irreducible curves for all $e,j$. Furthermore, there is an exact sequence
$$ 0\to \mathcal{O}_{T_{eR_s}}(-C_{e,j,X''_{R_s}})\to \mathcal{O}_{T_{eR_s}}\to \mathcal{O}_{C_{e,j,X''_{R_s}}}\to 0.$$
Note that $T_{eR_s}$ is a regular scheme, and thus $\mathcal{O}_{T_{eR_s}}(-C_{e,j,X''_{R_s}})$ is a locally free sheaf on $\mathcal{O}_{T_{eR_s}}$. Taking local cohomology gives us that for any (not necessarily closed) point $x\in X''_{R_s}$,
$$ H^i_x\left(X,\mathcal{O}_{C_{e,j,X''_{R_s}}}\right)=0\text{ for any }0\le i<\dim R_s+1-\dim \overline{\{x\}}.$$
Therefore, $C_{e,j,X''_{R_s}}$ is a Cohen-Macaulay scheme. Note that $R_s$ is a regular local ring, $\dim C_{e,j,X''_{R_s}}=\dim R_s+1$ because $C_{e,j,X''_{R_s}}\subseteq T_{eR_s}$ is a prime exceptional divisor and $\dim T_{eR_s}=\dim R_s+2$. Hence, by the miracle flatness (cf. \cite[Lemma 00R4]{Stacks}), $C_{e,j,X''_{R_s}}$ is a flat $R_s$-scheme. We constructed $C_{j,X''}$ and $C_{j,X''_{R_s}}$ as we want.

\smallskip

Let us prove (a). We may assume $X$ to be projective after compactification. Define $H_{e\overline{s}}:=H_{eR_s}|_{\overline{s}}, L_{e\overline{s}}:=L_{eR}|_{\overline{s}}$, and
$$ U:=|H_{3\overline{s}}|\times \cdots\times |H_{e\overline{s}}|\times |L_{(e+1)\overline{s}}|\times \cdots \times |L_{n\overline{s}}|.$$
Then any general $u=(H'_3,\cdots,H'_e,L'_{e+1},\cdots,L'_n)\in U$ determines a surface $T_{e,u,\overline{s}}$. Let $E$ be a prime $(f_{\overline{s}}\circ f'_{\overline{s}})$-exceptional divisor on $X''_{\overline{s}}$ such that $(f_{\overline{s}}\circ f'_{\overline{s}})(E)$ has codimension $e$ in $X_{\overline{s}}$, and let us define
$$ B(u):=(f_{\overline{s}}\circ f'_{\overline{s}})(E)\cap L'_{e+1}\cap \cdots \cap L'_n.$$
Then $B(u)$ is a finite disjoint union of closed points on $X_{\overline{s}}$,
$$ T_{e,u,\overline{s}}\cap E=H'_3\cap \cdots \cap H'_e\cap ((f_{\overline{s}}\circ f'_{\overline{s}})|_E)^{-1}(B(u)),$$
and thus $T_{e,u,\overline{s}}\cap E$ is a finite disjoint union of curves that are $(f_{\overline{s}}\circ f'_{\overline{s}})$-exceptional for a general $u\in U$. Moreover, for every $e,j$, $C_{e,j,\overline{s}}\subseteq E$ for some prime $(f_{\overline{s}}\circ f'_{\overline{s}})$-exceptional divisor $E$ on $X''_{\overline{s}}$. 

\smallskip

Since $T_{e\overline{s}}\cap E$ is the intersection of 
$$E,H_{3\overline{s}},\cdots,H_{e\overline{s}},(f_{\overline{s}}\circ f'_{\overline{s}})^{-1}_*L_{(e+1)\overline{s}},\cdots,(f_{\overline{s}}\circ f'_{\overline{s}})^{-1}_*L_{n\overline{s}},$$
and since
$$ E+H_{3\overline{s}}+\cdots+H_{e\overline{s}}+(f_{\overline{s}}\circ f'_{\overline{s}})^{-1}_*L_{(e+1)\overline{s}}+\cdots+(f_{\overline{s}}\circ f'_{\overline{s}})^{-1}_*L_{n\overline{s}}$$
is snc by (\ref{snc}), we obtain that the Hilbert polynomial of $T_{e\overline{s}}\cap E$ is the same as the Hilbert polynomial of $T_{e,u,\overline{s}}$ for a general $u\in U$ (if we pick a very ample line bundle $\mathcal{O}(1)$ on $X''_{\overline{s}}$). In fact, if we take the Koszul complex of $\mathcal{O}_{T_{e,u,\overline{s}}\cap E}$, then we have that the Hilbert polynomial is
$$ p(m)=\sum_{J\subseteq \{1,\cdots,n-1\}}(-1)^{|J|}\chi\left(X''_{\overline{s}},\mathcal{O}(m)\left(-\sum_{j\in J}D_j\right)\right),$$
where $D_1=E, D_2=H_{3\overline{s}},\cdots,D_{e-1}=H_{e\overline{s}},D_e=L_{(e+1)\overline{s}},\cdots,D_{n-1}=L_{n\overline{s}}$.

\smallskip

Therefore there is a flat family of $1$-dimensional schemes $W\to U'$ and $0\in U'$ with $W_0=T_{e\overline{s}}\cap E$ and $W_u=T_{e,u,\overline{s}}\cap E$, where $U'\subseteq U$ is a non-empty open subscheme (cf. \cite[Theorem 3.9.9]{Har77}). If we consider the irreducible components of the flat family, then we obtain that $C_{e,j,X''_{\overline{s}}}$ are movable curves in codimension $1$ on $X''_{\overline{s}}$. By the same argument, $C_{e,j,X''}$ is movable in codimension $1$ inside $X''$ for every $e,j$.

\smallskip

Let us prove (b). Note that the proof is almost the same as \cite[Proof of Lemma 14]{dFDTT15}. Suppose $E$ is a prime exceptional divisor on $X''$, and let
$$ B:=(f\circ f')(E)\cap L_{e+1}\cap \cdots\cap L_n.$$
Note that
$$ T_e\cap E=H_3\cap \cdots\cap H_e\cap ((f\circ f')|_E)^{-1}(B).$$
Let $c$ be the codimension of $(f\circ f')(E)$.

\begin{itemize}
    \item[(i)] If $c>e$, then $B$ is empty and $E$ is disjoint from $T_e$. In particular, $E\cdot C_{e,j}=0$ for all $j$.
    \item[(ii)] If $c=e$, then $B$ is zero-dimensional, and $(f\circ f')^{-1}(B)$ is a union of general fibers of $(f\circ f')|_E$. In particular, if $T_e\cap E$ is nonempty, then any irreducible component of $T_e\cap E$ is $C_{e,j,X''}$ for some $j$.
    \item[(iii)] If $c<e$, then $B$ is an irreducible set of positive dimension, and $((f\circ f')|_E)^{-1}(B)$ is irreducible as well. Then $T_e\cap E$ is an irreducible curve and is not $(f\circ f')$-exceptional. In particular, $E\cdot C_{e,j,X''}\ge 1$ for all $j$.
\end{itemize}

Let $D$ be a Cartier divisor on $X''$ such that $(f\circ f')_*D\le 0$, and $D\cdot C_{e,j,X''}\ge 0$ for all $e,j$. Write
$$ D=D_1+\cdots+D_n,$$
where each $D_e$ is supported exactly on the prime components of $D$ whose image in $X$ has codimension $e$. We will show $D_e\le 0$ for every $e$ by induction. If $e=1$, then the claim is trivial.

\smallskip

Assume $e>1$. By the construction of $H_{3R},\cdots,H_{nR}$ and $L_3,\cdots,L_n$, we obtain that $T_e$ intersects properly with every component of $D_e$, and if $E$ is an prime component of $D_e$, then $T_e\cap E$ is a general complete intersection in $((f\circ f')|_E)^{-1}(B)$. If $E'$ is a different prime component of $D_e$, and $B':=\pi(E')\cap L_{e+1}\cap \cdots \cap L_n$, then $((f\circ f')|_E)^{-1}(B)$ and $((f\circ f')|_{E'})^{-1}(B')$ do not have any components in common, since the hyperplane sections $L_{e+1},\cdots,L_n$ are very general (recall the fact that the number of sequences of toroidal blowups are countable).

\smallskip

Thus, $E|_{T_e}$ and $E'|_{T_e}$ share no common component. Therefore, it is enough to show $D_e|_{T_e}\le 0$. Note that by (ii), we see that the prime components of $D_e|_{T_e}$ are of the form $C_{e,j,X''}$ for some $j$. The use of the Negativity lemma for surfaces shows it suffices to show
\begin{equation} \label{>}
D_e|_{T_e}\cdot C_{e,j,X''}\ge 0\text{ for all }j.
\end{equation}
Since, by (i), $D_{e'}\cdot C_{e,j,X''}=0$ whenever $e'>0$, and
$$ D\cdot C_{e,j,X''}=(D_1+\cdots+D_e)\cdot C_{e,j,X''}=(D_1+\cdots+D_{e-1})\cdot C_{e,j,X''}+D_e\cdot C_{e,j,X''} $$
By induction, $D_1,\cdots,D_{e-1}\le 0$, and by (iii), we know that $C_{e,j,X''}$ is not contained in the support of $D_2+\cdots+D_{e-1}$. Moreover, $C_{e,j,X''}$ is movable in codimension $1$ and $(f\circ f')$-exceptional, the locus spanned by its deformations is $(f\circ f')$-exceptional. Hence, the intersection of $C_{e,j,X''}$ with each of the components of $D_1$ is non-negative. Therefore,
$$ (D_1+\cdots+D_{e-1})\cdot C_{e,j,X''}\le 0.$$
By the hypothesis in the lemma, $D\cdot C_{e,j,X''}\ge 0$, so we must have
$$ D_e|_{T_e}\cdot C_{e,j,X''}=D_e\cdot C_{e,j,X''}\ge 0,$$
and it is (\ref{>}).

\end{proof}

Let us define the following:

\begin{definition} \label{numerically}
Let $X$ be a normal variety over $\C$, and let $D$ be a $\Q$-Weil divisor on $X$.
\begin{itemize}

\item[(a)] Let us say $D$ is an \emph{asymptotically flat divisor} on $X$ if the following is satisfied: Suppose $\left(X',E:=\sum^r_{i=1}E_i\right)$ is a log smooth model over $X$, and $\eta$ is the generic point of a connected component of the intersection of $k$ prime components of $E$. Then there is a finitely generated $\Z$-algebra $R$, models $X_R$, $D_R$, and $(X'_R,E_R)$ of $X,D,(X',E)$ respectively and a Zariski-dense subset $S'\subseteq \Spec R$ such that for any closed point $s\in S'$ and any effective $\Q$-Weil divisor $\Delta'$ on $X_{\overline{s}}$ with $D_{\overline{s}}+\Delta'$ $\Q$-Cartier, 
$$ \nu_{(X',E),\eta,\alpha}(D)\le \nu_{(X'_{\overline{s}},E_{\overline{s}}),\eta_{\overline{s}},\alpha}(D_{\overline{s}}+\Delta') \text{ for all }\alpha\in \R^k_{\ge 0}.$$

    \item[(b)] We say $D$ is \emph{numerically $\Q$-Cartier} if there is a proper birational morphism $f:X'\to X$ with a $\Q$-Cartier divisor $D^{\mathrm{num}}$ on $X'$ such that $f_*D^{\mathrm{num}}=D$ and $D^{\mathrm{num}}$ is $f$-numerically trivial.
    \item[(c)] We say \emph{infinitely many reduction mod $p$ of $D$ is $\Q$-Cartier} if there is a model $X_R$ and $D_R$ of $X$ and $D$ over a finitely generated $\Z$-algebra $R$ respectively, and a Zariski-dense subset $S'\subseteq \Spec R$ such that $mD_{\overline{s}}$ is $(S_3)$ for every positive integer $m$ and every $s\in S'$.
\end{itemize}
\end{definition}

Note that the notion, \emph{numerically $\Q$-Cartier}, is first introduced in \cite{BdFFU15}.

\begin{proof}[Proof of Proposition \ref{GG}]
Let $\left(X',E:=\sum^r_{i=1}E_i\right)$ and $\alpha\in \Q^k_{\ge 0}$ be a log smooth model of $X$ and a tuple respectively. Let $X_R$, $D_R$, and $(X'_R,E_R)$ be models of $X$, $D$ and $(X',E)$ over a finitely generated $\Z$-algebra $R$. Let $f':X''\to X'$ be a toroidal blowup of $(X',E)$ such that the center of $\nu_{(X',E),\eta,\alpha}$ is a prime divisor $F$ on $X''$. Then there is a model $f'_R:X''_R\to X'_R$ of $f'$ and $F_R$ of $F$ over $R$.

\smallskip

We will prove the condition of asymptotic flatness only for $\alpha\in \Q^k_{\ge 0}$: the rest can be proved by Lemma \ref{youareeasy'}.

\smallskip

\noindent \textbf{Proof of (a)}: Let us consider \cite[Lemma 6.2]{CEMS18}. Then we obtain that there is a positive integer $n$ such that the natural morphism
$$ \bigoplus_{m\ge 0}\mathcal{O}_{X_R}(-mnD_R)\bigg|_{X_{\overline{s}}}\to \bigoplus_{m\ge 0}\mathcal{O}_{X_{\overline{s}}}(-mnD_{\overline{s}})$$
is an isomorphism. Considering the inclusions $\mathcal{O}_{X_{\overline{s}}}\subseteq \mathcal{O}_{X''_{\overline{s}},F_{\overline{s}}}\subseteq k(X_{\overline{s}})$ and $\mathcal{O}_{X}\subseteq \mathcal{O}_{X'',F}\subseteq k(X)$, we obtain that a generator of $\mathcal{O}_{X_{\overline{s}}}(mnD_{\overline{s}})\cdot \mathcal{O}_{X''_{\overline{s}},F_{\overline{s}}}$ in a discrete valuation ring $\mathcal{O}_{X''_{\overline{s}},F_{\overline{s}}}$ can be lifted to $\mathcal{O}_{X}(mnD)\cdot \mathcal{O}_{X'',F}$, and therefore by Lemma \ref{detail},
$$ \nu^{\natural}_{(X'_{\overline{s}},E_{\overline{s}}),\alpha}(mnD_{\overline{s}})\ge \nu^{\natural}_{(X',E),\alpha}(mnD),$$
and letting $m\to \infty$ gives us the assertion.

\smallskip

\noindent \textbf{Proof of (b)}: Note that there is a $\Q$-Cartier divisor $D^{\mathrm{num}}$ on $X''$ such that 
$$(f\circ f')_*D^{\mathrm{num}}=D$$
and $D^{\mathrm{num}}$ is $(f\circ f')$-numerically trivial (cf. \cite[Proposition 5.3]{BdFFU15}), and there is a model $D^{\mathrm{num}}_R$ of $D^{\mathrm{num}}$ over $R$.

\smallskip

Let $\Delta'$ be an effective $\Q$-Weil divisor on $X$ with $D_{\overline{s}}+\Delta'$ $\Q$-Cartier. Let us define
$$ F':=-(f_{\overline{s}}\circ f'_{\overline{s}})^*(D_{\overline{s}}+\Delta')+D^{\mathrm{num}}_{\overline{s}}+(f_{\overline{s}}\circ f'_{\overline{s}})^{-1}_*\Delta',$$
and let us use the notations in Lemma \ref{듸듸<3<3}. Since $F'$ is $(f_{\overline{s}}\circ f'_{\overline{s}})$-exceptional, there is an $(f\circ f')$-exceptional divisor $F$ on $X''$ with $F_{\overline{s}}=F'$. Since $C_{j,X''_{\overline{s}}}$ is movable in codimension $1$ and $(f_{\overline{s}}\circ f'_{\overline{s}})$-exceptional, we have 
$$(f_{\overline{s}}\circ f'_{\overline{s}})^{-1}_*\Delta'\cdot C_{j,X''_{\overline{s}}}\ge 0.$$
Hence,
$$F\cdot C_{j,X''}=F'\cdot C_{j,X''_{\overline{s}}}\ge 0,$$
and therefore, by Lemma \ref{듸듸<3<3}, $F\le 0$. Thus, $F'\le 0$, and
$$ \nu_{(X',E),\eta,\alpha}(D)\le \nu_{(X'_{\overline{s}},E_{\overline{s}}),\eta_{\overline{s}},\alpha}(D_{\overline{s}}+\Delta').$$

\smallskip

\noindent \textbf{Proof of (c)}: Note that since $X$ has rational singularities, $D$ is $\Q$-Cartier outside of a codimension $\ge 3$ closed subscheme (cf. \cite[Proposition B.7]{CS23}). Let $S'\subseteq \Spec R$ be a Zariski-dense subset of $\Spec R$ such that $mD_{\overline{s}}$ is $(S_3)$ for every positive integer $m$ and every $s\in S'$, and $Z\subseteq X$ be a codimension $\ge 3$ closed subscheme of $X$ such that $D|_{X\setminus Z}$ is $\Q$-Cartier.

\smallskip

After shrinking $R$, we may assume that there exists a model $Z_R\subseteq X_R$ of $Z\subseteq X$, and we may further assume $D_R|_{X_R\setminus Z_R}$ to be as Cartier after multiplying by some positive integer. Note that
$$ \imath_{\overline{s}*}\left(\mathcal{O}_{X_R}(mD_R)|_{X_{\overline{s}}\setminus Z_{\overline{s}}}\right)=\mathcal{O}_{X_{\overline{s}}}(mD_{\overline{s}})$$
for every $m$ and every $s\in S'$. Hence, by \cite[Theorem 10.73]{Kol23} and the generic flatness (cf. \cite[Proposition 052A]{Stacks}), we obtain that $\mathcal{O}_{X_R}(mD_R)$ is flat over $R$ for every $m$. Considering \cite[Theorem 9.17]{Kol23}, we have a natural isomorphism
$$ \mathcal{O}_{X_R}(mD_R)|_{\overline{s}}\cong \mathcal{O}_{X_{\overline{s}}}(mD_{\overline{s}})$$
for every $m$, and therefore we may lift an element of $\mathcal{O}_{X_{\overline{s}}}(mD_{\overline{s}})\cdot \mathcal{O}_{X''_{\overline{s}},F_{\overline{s}}}$ that generates the range of $\nu_{(X'_{\overline{s}},E_{\overline{s}}),\eta_{\overline{s}},\alpha}$ to $\mathcal{O}_{X}(mD)\cdot \mathcal{O}_{X'',F}$. Hence, by Lemma \ref{detail}, we obtain the assertion.

\end{proof}

\begin{example} \label{exex}
Let me give an example of a Weil divisor $D$ that is not $\Q$-Cartier and $mD$ is $(S_3)$ in infinitely many reductions mod $p$. Let $X$ be a smooth projective variety such that $-K_X$ is ample, and
$$ H^1(X_p,\mathcal{O}_{X_p}(mK_{X_p}))=0$$
for every integer $m$ and every prime $p$, where $X_p$ is the reduction mod $p$ of $X$ (simply put $X:=\P^2$). Let $a,b,m,n$ be integers. Then, by the Künneth formula (cf. \cite[Lemma 0BED]{Stacks}),
\begin{equation} \label{ex1}
H^1\left(X_p\times X_p,\mathcal{O}_{X_p\times X_p}((m+an)K_{X_p},(m+bn)K_{X_p})\right)=0.
\end{equation}
Let us assume $a\ne 1,b$ to be coprime positive integers, and let us define
$$ \mathcal{L}_p:=\mathcal{O}_{X_p\times X_p}(-aK_{X_p},-bK_{X_p}).$$

\smallskip

We may consider the affine cone
$$ S_p:=\Spec \bigoplus_{m\ge 0}H^0(X_p\times X_p,\mathcal{L}^m_p).$$
We note the following:
\begin{itemize}
    \item Because of the fact $\mathrm{Cl}(S_p)=\mathrm{Pic}(X_p\times X_p)/(\mathcal{O}_{X_p\times X_p}(-aK_{X_p},-bK_{X_p}))$ (cf. \cite[Proposition 3.14]{Kol13}), $\mathcal{O}_{X_p\times X_p}(K_{X_p},K_{X_p})$ maps to a non-zero element in $\mathrm{Cl}(S_p)$. Moreover, $\Pic(S_p)=0$ by \cite[Proposition 3.14]{Kol13} again. Hence, $-K_{S_p}$ is not $\Q$-Cartier.
    \item By (\ref{ex1}) and \cite[Remark 3.12]{Kol13}, we know that $-mK_{S_p}$ is $(S_3)$ for any integer $m$.
\end{itemize}

\end{example}

\section{Log discrepancies in sense of de Fernex-Hacon} \label{4444}

In this section, we introduce the notion of \emph{log discrepancy} for a couple $(X,\Delta)$ that is first introduced in \cite[Chapter 7]{dFH09}. We keep the standing notation of Notation \ref{nota}.

\begin{definition} \label{see}
Let $(X,\Delta)$ be a couple, $f:X'\to X$ a proper birational morphism between normal varieties, and $E$ a prime divisor on $X'$. We define
$$ A_{X,\Delta,m}(E):=\mathrm{ord}_E(K_{X'})+1-\frac{1}{m}\mathrm{ord}^{\natural}_E(m(K_X+\Delta)),$$
and
$$ A_{X,\Delta}(E):=\mathrm{ord}_E(K_{X'})+1-\mathrm{ord}_E(K_X+\Delta)=\lim_{m\to \infty}A_{X,\Delta,m}(E).$$
\end{definition}

Note that the notion coincides with the usual log discrepancy function if $(X,\Delta)$ is a pair. A main theorem of \cite{dFH09} is that we have a criterion of being of klt type using the notion of log discrepancy.

\begin{theorem}[{cf. \cite[Theorem 1.2]{dFH09}}] \label{dFH09}
Let $(X,\Delta)$ be a couple over a characteristic $0$ field. Then $(X,\Delta)$ is of klt type if and only if there exists a positive integer $m$ such that $A_{X,\Delta,m}(E)>0$ for every prime divisor $E$ over $X$.
\end{theorem}

We can prove the following lemma, which can be viewed as a weak version of our main theorem.

\begin{lemma} \label{weak}
Let $(X,\Delta)$ be an affine couple with $K_X+\Delta$ asymptotically flat that is of strongly $F$-regular type, and let $E$ be a prime divisor over $X$. Then $A_{X,\Delta}(E)>0$.
\end{lemma}

\begin{proof}
Let $f:X'\to X$ be a log resolution of $(X,\Delta)$ with $E$ a prime divisor on $X$, and let $(X_R,\Delta_R)$, $f_R:X'_R\to X_R$ and $E_R$ be models of $(X,\Delta)$, $f$, $E$ respectively over a finitely generated $\Z$-algebra.

\smallskip

Considering the definition of asymptotic flatness, after shrinking $R$, we obtain
$$ \mathrm{ord}_{E}(K_{X}+\Delta)\le \mathrm{ord}_{E_{\overline{s}}}(K_{X_{\overline{s}}}+\Delta_{\overline{s}}+\Delta')$$
for any effective $\Q$-Weil divisor $\Delta'$ on $X_{\overline{s}}$ with $K_{X_{\overline{s}}}+\Delta_{\overline{s}}+\Delta'$ $\Q$-Cartier, and
$$ 
\begin{aligned}
A_{X,\Delta}(E)&=\mathrm{ord}_E(K_{X'})+1-\mathrm{ord}_E(K_X+\Delta)
\\ &\ge \mathrm{ord}_{E_{\overline{s}}}(K_{X'_{\overline{s}}})+1-\mathrm{ord}_{E_{\overline{s}}}(K_{X_{\overline{s}}}+\Delta_{\overline{s}}+\Delta')
\\ &=A_{X_{\overline{s}},\Delta_{\overline{s}}+\Delta'}(E_{\overline{s}}).
\end{aligned}
$$
for all closed points $s\in \Spec R$. Fix a closed point $s\in \Spec R$. By Theorem \ref{SS10}, we can choose $\Delta'$ such that $(X_{\overline{s}},\Delta_{\overline{s}}+\Delta')$ is strongly $F$-regular. We know that $(X_{\overline{s}},\Delta_{\overline{s}}+\Delta')$ is klt, and thus $A_{X_{\overline{s}},\Delta_{\overline{s}}+\Delta'}(E_{\overline{s}})>0$. Hence,
$$ A_{X,\Delta}(E)\ge A_{X_{\overline{s}},\Delta_{\overline{s}}+\Delta'}(E_{\overline{s}})>0,$$
and we proved the assertion.
\end{proof}

\section{Proof of the main theorem} \label{55555}
In this section, we prove the main theorem. As explained in the introduction, we will use the Jonsson-Mustață conjecture, and to make a setting to use that, we will take a small $\Q$-factorialization to a strongly $F$-regular type couple. We keep the standing notation of Notation \ref{nota}.

\begin{definition} \label{Q}
Let $(X,\Delta)$ be a couple, and let $f:X'\to X$ be a small morphism. We say that $f$ is a \emph{small $\Q$-factorializtion} of $(X,\Delta)$ if
\begin{itemize}
    \item[(a)] $X'$ is $\Q$-factorial,
    \item[(b)] $K_{X'}+f^{-1}_*\Delta=f^*(K_X+\Delta)$, and
    \item[(c)] $(X',f^{-1}_*\Delta)$ is klt.
\end{itemize}
\end{definition}

\begin{lemma}[{cf. \cite[Lemma 3.1]{CLZ25}}] \label{strong1}
Let $(X,\Delta)$ be an affine couple over $\C$ with $K_X+\Delta$ asymptotically flat that is of strongly $F$-regular type. Then there exists a small $\Q$-factorialization of $(X,\Delta)$.
\end{lemma}

\begin{proof}
Note that the lemma can be viewed as a reduction mod $p$ version of \cite[Lemma 3.1]{CLZ25}.

\smallskip

Let $f':X''\to X$ be a log resolution of $(X,\Delta)$, $m_0$ a positive integer, and let
$$\Delta'':=K_{X''}-\frac{1}{m_0}f^*(m_0(K_X+\Delta)).$$
By Lemma \ref{weak}, for sufficiently large $m_0$, we know that all coefficients of $\Delta''$ are $<1$. Let $E_1,\cdots,E_r$ be prime $f'$-exceptional divisors on $X''$ such that $A_{X,\Delta}(E_i)\ge 1$, and $E'_1,\cdots,E'_{r'}$ prime $f'$-exceptional divisors on $X''$ such that $A_{X,\Delta}(E'_{j})<1$ such that $\mathrm{Exc}(f)=\bigcup_i E_i\cup \bigcup_j E'_j$.

\smallskip

Let $m$ be a positive integer, and for each positive integer $m$. Suppose $D$ is a Weil divisor on $X$ such that $D\sim m_0(K_X+\Delta)$ and $-D$ is effective. Such a $D$ exists because $X$ is affine. Let $(\Delta'')^{>0}$ be the positive part of $\Delta''$, and $(\Delta'')^{\le 0}:=\Delta''-(\Delta'')^{>0}.$

\smallskip

There exists $\varepsilon>0$ such that the coefficients of $\Delta'''+\varepsilon\sum_iE_i$ are $<1$, and
$$ K_{X''}+(\Delta'')^{>0}+\varepsilon \sum_iE_i\sim_{\Q} \frac{1}{m_0}f^{\natural}D_{m_0}-(\Delta'')^{\le 0}+\varepsilon \sum_i E_i.$$
Let us run a $\left(K_{X''}+(\Delta'')^{>0}+\varepsilon \sum_i E_i\right)$-MMP over $X$. Suppose $g:X''\dashrightarrow X'''$ is a step of $\left(K_{X''}+(\Delta'')^{>0}+\varepsilon \sum_iE_i\right)$-MMP over $X$, and let $h:X'''\to X$ be the structure morphism. If $g$ is a divisorial contraction, then since $X'''$ is $\Q$-factorial and $(h\circ g)^{\natural}D-g^{\natural}h^{\natural}D$ is effective $g$-exceptional (cf. Remark \ref{natural} (d)), we obtain $g_*(f')^{\natural}D=h^{\natural}D$. If $g$ is a flip, then by the definition of natural pullback, $g_*(f')^{\natural}D=h^{\natural}D$. Thus, in any case, $g_*(f')^{\natural}D=h^{\natural}D$. Note that in any choice of $D$, we can consider only one MMP by Remark \ref{natural} (b).

\smallskip

Thus, we can construct a proper birational morphism $f:X'\to X$ such that $\frac{1}{m_0}f^{\natural}D+F$ is a limit of movable divisors over $X$ by running a $\left(K_{X''}+(\Delta'')^{>0}+\varepsilon \sum_i E_i\right)$-MMP over $X$, where $F$ is the strict transform of $-(\Delta'')^{\le 0}+\varepsilon \sum_iE_i$. For any prime divisor $E'$ on $X'$ that is $f$-exceptional, $\left.\left(\frac{1}{m_0}f^{\natural}D+F\right)\right\vert_{E'}$ is pseudo-effective over $f(E')$, and thus by the general negativity lemma (cf. \cite[Lemma 3.3]{Bir12}), we deduce $-\frac{1}{m_0}f^{\natural}D-F$ is effective.

\smallskip

Note that by Lemma \ref{zzm}, we may assume
$$
\mathrm{ord}^{\natural}_E(D)=0\text{ for every prime }f\text{-exceptional divisor }E.
$$
Under the choice of $D$, with Remark \ref{natural} (c), the condition that $-\frac{1}{m_0}f^{\natural}D-F$ is effective is possible only when $F=0$ because of the fact that $\mathrm{mult}_E\left(-\frac{1}{m_0}f^{\natural}D-F\right)\le 0$ for every prime $f$-exceptional divisor $E$ on $X'$.

\smallskip

Note that $f$ is a small morphism since $\Supp \left(-(\Delta'')^{\le 0}+\varepsilon \sum_i E_i\right)=\mathrm{Exc}(f')$, and $-(\Delta'')^{\le 0}+\varepsilon \sum_i E_i$ is contracted by the MMP. Moreover,
$$ K_{X'}+\Delta'=\frac{1}{m_0}f^{\natural}(m_0(K_X+\Delta)).$$
We have run an MMP from a smooth variety, and thus $X'$ is $\Q$-factorial, and it is (a) in Definition \ref{Q}. Since $f$ is small, by Remark \ref{natural} (e), we know that
$$K_{X'}+\Delta'=\frac{1}{m_0}f^{-1}_*(m_0(K_X+\Delta))=f^*(K_X+\Delta),$$
and it is (b) in Definition \ref{Q}. Furthermore, the strict transform of $(\Delta'')^{>0}+\varepsilon \sum_i E_i$ to $X'$ is $f^{-1}_*\Delta$, thus $(X',f^{-1}_*\Delta)$ is klt, and it is (c) in Definition \ref{Q}. We proved all criteria of Definition \ref{Q}.
\end{proof}

\begin{remark}
We hope our argument using natural valuation can remove the lc complement condition in \cite[Lemma 3.1]{CLZ25} if we allow a generic finite base change on $T$.
\end{remark}

\begin{lemma} \label{BLX}
Let $(X,\Delta)$ be a couple over $\C$, and let $f:X'\to X$ be a small $\Q$-factorialization of $(X,\Delta)$, and define 
$$\Phi:=\{m\in \Z_{\ge 0}\,|\,m(K_X+\Delta)\text{ is Weil and }m(K_{X'}+f^{-1}_*\Delta)\text{ is Cartier}\}$$
and
$$ \mathfrak{a}^{X'}_m:=\mathcal{O}_X(-m(K_X+\Delta))\cdot \mathcal{O}_{X'}(m(K_{X'}+f^{-1}_*\Delta)).$$
Then $\mathfrak{a}^{X'}_{\bullet}:=\{\mathfrak{a}^{X'}_m\}_{m\in \Phi}$ is a graded sequence of ideals in $\mathcal{O}_{X'}$, and $\mathrm{lct}(X',f^{-1}_*\Delta,\mathfrak{a}^{X'}_{\bullet})>1$ if and only if $(X,\Delta)$ is of klt type.
\end{lemma}

\begin{proof}
Let us check $\mathfrak{a}^{X'}_{\bullet}:=\{\mathfrak{a}^{X'}_m\}_{m\in \Phi}$ satisfies the properties we want.

\smallskip

\begin{itemize}
    \item[(i)] $\mathfrak{a}^{X'}_m$ is actually an ideal sheaf in $\mathcal{O}_{X'}$.
\end{itemize}

\smallskip

Observe that $f$ is a small morphism, and therefore there exists an open subscheme $U\subseteq X'$ with $\mathrm{codim}_X(X\setminus U)\ge 2$ and $\mathfrak{a}^{X'}_m|_U=\mathcal{O}_{U}$. Hence, $\left(\mathfrak{a}^{X'}_m\right)^{\vee\vee}=\mathcal{O}_{X'}$, and hence the double dual induces the inclusion $\mathfrak{a}^{X'}_m\subseteq \mathcal{O}_{X'}$.

\smallskip

\begin{itemize}
    \item[(ii)] The sequence $\mathfrak{a}^{X'}_{\bullet}:=\{\mathfrak{a}^{X'}_m\}_{m\in \Phi}$ is a graded sequence of ideals in $\mathcal{O}_{X'}$.
\end{itemize}

\smallskip

We observe that
$$
\begin{aligned}
    \mathcal{O}_{X}(-m(K_X+\Delta)) &\cdot \mathcal{O}_{X'}(-n(K_X+\Delta)) 
    \\ &\subseteq \mathcal{O}_{X'}(-m(K_X+\Delta)-n(K_X+\Delta))
    \\ &=\mathcal{O}_{X'}(-(m+n)(K_X+\Delta)).
\end{aligned}
$$
By multiplying both sides by $\mathcal{O}_{X'}((m+n)(K_{X'} + f^{-1}_*\Delta))$, we verify the claim.

\smallskip

\begin{itemize}
    \item[(iii)] $\mathrm{lct}(X',f^{-1}_*\Delta,\mathfrak{a}^{X'}_{\bullet})>1$ if and only if $(X,\Delta)$ is klt.
\end{itemize}

\smallskip

To demonstrate the direct implication, we refer to \cite[Lemma 1.60]{Xu25}, we have that for some $m\in \Phi$, $\mathrm{lct}(X',f^{-1}_*\Delta,\mathfrak{a}^{X'}_m)>1$. Observe that for any prime divisor $E$ on $X'$, we have:
$$ \frac{A_{X',f^{-1}_*\Delta}(E)}{\mathrm{ord}_E\mathfrak{a}^{X'}_m}>\frac{1}{m}.$$
Consequently, for any prime divisor $E$ on $X''$, the following holds:
\begin{equation} \label{very hard}
A_{X',f^{-1}_*\Delta}(E)-\frac{1}{m}\mathrm{ord}^{\natural}_E(m(K_X+\Delta))+\mathrm{ord}_E(K_{X'}+f^{-1}_*\Delta)>0.
\end{equation}
Observe that the expression on the left is given by
$$\mathrm{ord}_E(K_{X''})+1-\frac{1}{m}\mathrm{ord}^{\natural}_E(m(K_X+\Delta))=A_{X,\Delta,m}(E).$$
Thus, (\ref{very hard}) implies that $(X,\Delta)$ is of klt type (cf. Theorem \ref{dFH09}). To prove the converse, simply reverse the steps of the proof.
\end{proof}

\begin{remark}
If $(X,\Delta)$ is a pair, then $\mathfrak{a}^{X'}_m=\mathcal{O}_X$ for any $m\in \Phi$, and $\mathrm{lct}(X',f^{-1}_*\Delta,\mathfrak{a}^{X'}_{\bullet})=\infty$.
 \end{remark}

Let us prove a lemma about Diophantine approximation.

\begin{lemma} \label{Diophantus of Alexandria}
Let $\alpha=(\alpha_1,\cdots,\alpha_r)\in \R^r_{\ge 0}$ be a non-negative real vector, and fix $\varepsilon>0$. Then there exists $\alpha'\in \Q^r_{\ge 0}$ and a positive integer $n$ such that:
\begin{itemize}
    \item $n\alpha \in \Z^r_{\ge 0}$, and
    \item $\|\alpha-\alpha'\|<\frac{\varepsilon}{n}$.
\end{itemize}
\end{lemma}

\begin{proof}
By the Weyl's criterion (cf. \cite[Theorem 1.6.1]{KN74}), there exists a positive integer $n$ such that 
$$n\alpha_i-\floor{n\alpha_i}<\frac{\varepsilon}{\sqrt{r}}\text{ for each }i.$$
Define $\alpha'_i:=\floor{n\alpha_i}$ and set
$$\alpha:=\left(\frac{\alpha'_1}{n},\cdots,\frac{\alpha'_r}{n}\right).$$
Then $\|n\alpha-n\alpha_i\|<\varepsilon$. Dividing by $n$, we get $\|\alpha-\alpha'\|<\frac{\varepsilon}{n}$. This proves the claim.
\end{proof}

Finally, we can prove our main theorem, Theorem \ref{5000}

\begin{proof}[Proof of Theorem \ref{5000}]
The converse direction can be proved by Theorem \ref{HaraWatanabe}, and we need to prove only the direct direction. Note that the proof follows the same structure as that of \cite[Theorem 1.3]{KK23} and \cite[Proposition 4.2]{CJKL25}.

\smallskip

To prove that, let us use Lemma \ref{BLX}. By Lemma \ref{strong1}, there exists a small $\Q$-factorialization $f:X'\to X$ of $(X,\Delta)$, and let $\mathfrak{a}^{X'}_{\bullet}$ be a graded sequence of ideals in $\mathcal{O}_{X'}$ constructed in Lemma \ref{BLX}.

\smallskip

By Theorem \ref{quasar}, there exists a quasi-monomial valuation $\nu_0\in \mathrm{Val}_X$ that computes $\mathrm{lct}(X',f^{-1}_*\Delta,\mathfrak{a}^{X'}_{\bullet})$. Let $f:(X'',E)\to X'$ be a log smooth model of $X'$ such that $\nu_0\in \mathrm{QM}(X'',E)$. Let $f_R:X'_R\to X_R$ and $(X''_R,E_R)\to X'_R$ be models of $f:X'\to X$ and $(X'',E)\to X'$ over a finitely generated $\Z$-algebra $R$ respectively.
\smallskip

We obtain
\begin{equation} \label{a} 
\nu_0(\mathfrak{a}^{X'}_{\bullet})=\nu_0(K_X+\Delta)-\nu_0(K_{X'}+f^{-1}_*\Delta).
\end{equation}
Let $E_1,\cdots,E_r$ be prime components of $E$ these define $\nu_0$, $\eta:=c_{X'}(\nu_0)$ and $\alpha\in \R^r_{>0}$ a tuple with $\nu_0=\nu_{(X'',E),\eta,\alpha}$.

\smallskip

The definition of asymptotic flatness gives us that after shrinking $R$, for any effective $\Q$-Weil divisor $\Delta'$ on $X_{\overline{s}}$ with $K_{X_{\overline{s}}}+\Delta_{\overline{s}}+\Delta'$ $\Q$-Cartier,
\begin{equation} \label{b} \nu_{(X''_{\overline{s}},E_{\overline{s}}),\eta_{\overline{s}},\alpha}(K_{X_{\overline{s}}}+\Delta_{\overline{s}}+\Delta')\ge \nu_{(X'',E),\eta,\alpha}(K_X+\Delta)\text{ for all closed points }s\in \Spec R.
\end{equation}
Moreover, by using Proposition \ref{disc 2}, we obtain
\begin{equation}\label{c}
A_{X'_{\overline{s}},f^{-1}_{\overline{s}*}\Delta_{\overline{s}}}(\nu_{(X''_{\overline{s}},E_{\overline{s}}),\eta_{\overline{s}},\alpha})=A_{X',f^{-1}_*\Delta}(\nu_{(X'',E),\eta,\alpha})\text{ for all closed points }s\in \Spec R.
\end{equation}
Set $\nu'_0:=\nu_{(X''_{\overline{s}},E_{\overline{s}}),\eta_{\overline{s}},\alpha}$, and fix a closed point $s\in \Spec R$.

\smallskip

To deduce $\mathrm{lct}(X',f^{-1}_*\Delta,\mathfrak{a}^{X'}_{\bullet})>1$ from (\ref{a}), (\ref{b}) and (\ref{c}), it suffices to show that for some $\Delta'$,
\begin{equation} \label{goal} A_{X'_{\overline{s}},f^{-1}_{\overline{s}*}\Delta_{\overline{s}}}(\nu'_0)-\nu'_0(K_{X_{\overline{s}}}+\Delta_{\overline{s}}+\Delta')+\nu'_0(K_{X'_{\overline{s}}}+f^{-1}_{\overline{s}*}\Delta_{\overline{s}})>0. 
\end{equation}

\smallskip

 Since $(X_{\overline{s}},\Delta_{\overline{s}})$ is strongly $F$-regular, Theorem \ref{SS10} ensures the existence of an effective $\Q$-Weil divisor $\Delta'$ on $X_{\overline{s}}$ such that the pair $(X_{\overline{s}},\Delta_{\overline{s}}+\Delta')$ is a strongly $F$-regular pair. Fix the $\Delta'$. Now, one can find a positive number $\varepsilon>0$ satisfying 
\begin{equation} \label{d}
    A_{X_{\overline{s}},\Delta_{\overline{s}}+\Delta'}(E')>\varepsilon
\end{equation}
for every prime divisor $E'$ over $X_{\overline{s}}$.

\smallskip

Let $\|\cdot \|$ be a metric on $\mathrm{QM}(X''_{\overline{s}},E_{\overline{s}})$ that induces the topology, and define
$$ \phi(\nu):=A_{X'_{\overline{s}},f^{-1}_{\overline{s}}\Delta_{\overline{s}}}(\nu)-\nu(K_{X_{\overline{s}}}+\Delta_{\overline{s}}+\Delta')+\nu(K_{X'_{\overline{s}}}+f^{-1}_{\overline{s}*}\Delta_{\overline{s}})$$
for each $\nu\in \mathrm{QM}_{\eta}(X''_{\overline{s}},E_{\overline{s}})$. Then
\begin{itemize}
    \item[(i)] if $f:X'''\to X'_{\overline{s}}$ is a proper birational morphism between normal varieties and $E'$ a prime divisor on $X''$, then by (\ref{d}), 
    $$\phi(\mathrm{ord}_{E'})=\mathrm{ord}_{E'}(K_{X'''})+1-\mathrm{ord}_{E'}(K_{X_{\overline{s}}}+\Delta_{\overline{s}}+\Delta')=A_{X_{\overline{s}},\Delta_{\overline{s}}+\Delta'}(E')>\varepsilon.$$
    \item[(ii)] $\phi$ is a convex function, and thus locally Lipschitz (cf. \cite[Theorem
2.1.22]{Hor07}).
\end{itemize}

By Lemma \ref{Diophantus of Alexandria}, for each $t>0$, there exist divisorial valuations $\nu_t\in \mathrm{QM}_{\eta}(X''_{\overline{s}},E_{\overline{s}})$ and a positive integer $n$ such that
\begin{itemize}
    \item[(iii)] $n\cdot \nu_t=\mathrm{ord}_{F_t}$ for some prime divisor $F_t$ over $X_{\overline{s}}$.
    \item[(iv)] $\|\nu'_0-\nu_t\|<\frac{t}{n}.$
\end{itemize}

By (ii), there exists $C>0$ such that for any $0<t\ll 1$, 
\begin{equation}\label{C}
|\phi(\nu'_0)-\phi(\nu_t)|<C\|\nu'_0-\nu_t\|.
\end{equation}
We see that for any $0<t\ll 1$,
$$ 
\begin{aligned}
\phi(n\cdot \nu'_0)&\ge \phi(n\cdot \nu_t)-|\phi(n\cdot \nu'_0)-\phi(n\cdot \nu_t)|
\\ &\ge \phi(\mathrm{ord}_{F_t})-C\cdot n\cdot \|\nu'_0-\nu_t\|& (1)
\\ &>\varepsilon-C\cdot t& (2)
\\ &>0,
\end{aligned}$$
where we used (iii) and (\ref{C}) in (1), and (i) and (iv) in (2). Hence, $\phi(\nu'_0)>0$ and (\ref{goal}) hold. We complete the proof.
\end{proof}


\begin{thebibliography}{dFDTT15}

\bibitem[Bir12]{Bir12}
C. Birkar, \emph{Existence of log canonical flips and a special LMMP}, Pub. Math. IHES. \textbf{115} (2012), 325--368.

\bibitem[Bir16]{Bir16}
C. Birkar,
\emph{Existence of flips and minimal models for 3-folds in char
              {$p$}},
Ann. Sci. \'Ec. Norm. Sup\'er. (4)
\textbf{49}
(2016),
no. 1,
{169--212}.

\bibitem[CS23]{CS23}
J. Carvajal-Rojas and A. St\"abler,
\emph{On the local \'etale fundamental group of {KLT} threefold
              singularities},
              {With an appendix by J\'anos Koll\'ar},
              {Algebr. Geom.}
              \textbf{10}
              (2023),
              no. 6,
              694--728.
              
        

\bibitem[Bir19]{Bir19}
C. Birkar, \emph{Anti-pluricanonical systems on {F}ano varieties}, Ann. of Math. (2), \textbf{190} (2019), no. 2, 345--463.

\bibitem[BCHM10]{BCHM10} C. Birkar, P. Cascini, C. Hacon, and J. McKernan, \emph{Existence of minimal models for varieties of log general type},  J. Amer. Math. Soc. \textbf{23} (2010), no. 2, 405--468.

\bibitem[BdFFU15]{BdFFU15}
S. Boucksom, T. de Fernex, C. Favre, and S. Urbinati,
\emph{Valuation spaces and multiplier ideals on singular varieties},
Recent advances in algebraic geometry,
London Math. Soc. Lecture Note Ser.,
vol. 417,
29--51,
Cambridge Univ. Press, Cambridge,
2015.

\bibitem[BLX22]{BLX22}
H. Blum, Y. Liu and C. Xu, \emph{Openness of K-semistability for Fano varieties}, Duke Math. J. \textbf{171} (2022), no. 13, 2753--2797.

\bibitem[BMP+20]{BMP+20}
B. Bhatt, L. Ma, Z. Patakfalvi, K. Schwede, K. Tucker, J. Waldron, and J. Witaszek,
\emph{Globally +-regular varieties and the minimal model program for threefolds in mixed characteristic},
Publ. Math. Inst. Hautes \'{E}tudes Sci.
\textbf{138}
(2023),
69--227.

\bibitem[BS12]{BS12}
M. Blickle and K. Schwede, \emph{$p^{-1}$-linear maps in algebra and geometry}, Springer, 123--205, Commutative Algebra: Expository Papers Dedicated to David Eisenbud on the Occasion of His 65th Birthday, (2012).

\bibitem[Can20]{Can20}
E. Canton, \emph{Berkovich log discrepancies in positive characteristic}, Pure Appl. Math. Q. \textbf{16} (2020), no. 5, 1465--1532.

\bibitem[CEMS18]{CEMS18}
A. Chiecchio, F. Enescu, L. Miller and K. Schwede, \emph{Test ideals in rings with finitely generated anti-canonical algebras}, J. Inst. Math. Jussieu. \textbf{17} (2018), no. 1, 171--206.

\bibitem[CJKL25]{CJKL25}
S. Choi, S. Jang, D. Kim, and D-W. Lee, \emph{A valuative approach to the anticanonical minimal model program}, arXiv preprint arXiv:2506.13637, (2025).

\bibitem[CLZ25]{CLZ25}
S. Choi, Z. Li and C. Zhou, \emph{Variation of cones of divisors in a family of varieties--Fano type case}, arXiv preprint arXiv:2504.04109 (2025).

\bibitem[dFH09]{dFH09}
T. De Fernex and C. Hacon, \emph{Singularities on normal varieties}, Compositio Math. \textbf{145} (2009), no. 2, 393--414.

\bibitem[dFDTT15]{dFDTT15}
T. de Fernex, R. Docampo, S. Takagi, and K. Tucher, \emph{Comparing multiplier ideals to test ideals on numerically $\Q$-Gorenstein varieties}, Bull. Lond. Math. Soc. \textbf{47} (2015), no. 2, 359--369.

\bibitem[Fuj17]{Fuj17}
O. Fujino,
\textit{Foundations of the minimal model program},
MSJ Memoirs,
vol. 35,
Mathematical Society of Japan,
Tokyo,
2017.

\bibitem[Har77]{Har77}
R. Hartshorne, \emph{algebraic geometry}, Graduate Texts in Mathematics, xvi+496, Springer-Verlag, New York-Heidelberg 1977.

\bibitem[Hoc22]{Hoc22}
M. Mochster, \emph{Tight closure and strongly $F$-regular rings}, Res. Math. Sci. \text{9} (2022), no. 3.

\bibitem[Hor07]{Hor07}
L. H{\"o}rmander, \emph{Notions of convexity}, Springer Science \& Business Media, 2007.

\bibitem[HH90]{HH90}
M. Hochster and C. Huneke, \emph{Tight closure, invariant theory, and the Brian{\c{c}}on-Skoda theorem}, J. Amer. Math. Soc. \textbf{3} (1990), no. 1, 31--116.

\bibitem[HH99]{HH99}
M. Hochster and C. Huneke, \emph{Tight closure in equal characteristic zero}, preprint, (1999).

\bibitem[HS06]{HS06}
C. Huneke, and I. Swanson, \emph{Integral closure of ideals, rings, and modules}, 226. xiv+431pp. London Mathematical Society Lecture Note Series, London, (2006).

\bibitem[HW02]{HW02}
N. Hara and K. Watanabe, \emph{{F}-regular and {F}-pure rings vs. log terminal and log
              canonical singularities}, J. Algebraic Geom. \textbf{11} (2002), no. 2, 363--392.

\bibitem[JM12]{JM12}
M. Jonsson and
M. Musta{\c{t}}{\u{a}},
\emph{Valuations and asymptotic invariants for sequences of ideals},
Ann. Inst. Fourier (Grenoble)
\textbf{62}
(2012),
no. 6,
2145--2209.

\bibitem[Kim25a]{Kim25a}
D. Kim:
\emph{Rational singularities and {$q$}-birational morphisms},
{Internat. J. Math.}
\textbf{36}
(2025),
no. 1,
{Paper No. 2450065}.

\bibitem[Kim25b]{Kim25b}
D. Kim, \emph{On volumes and the generic invariance of Fano type varieties}, arXiv preprint arXiv:2506.13603, (2025).

\bibitem[Kol08]{Kol08}
J. Koll\'ar, \emph{Exercises in the birational geometry of algebraic varieties}, arXiv preprint arXiv:0809.2579.

\bibitem[Kol13]{Kol13}
J. Koll{\'a}r,
\textit{Singularities of the minimal model program},
Cambridge Tracts in Mathematics,
vol. 200,
Cambridge University Press,
Cambridge,
(2013).

\bibitem[Kol23]{Kol23}
J. Koll\'ar, \emph{Families of varieties of general type}, Cambridge Tracts in Mathematics, vol. 231, With the collaboration of Klaus Altmann and S\'andor J.
              Kov\'acs, Cambridge University Press, Cambridge, 2023.

\bibitem[KK23]{KK23}
D. Kim and J. Kollar, \emph{Log canonical singularities of plurisubharmonic functions}, arXiv preprint arXiv:2312.16140, (2023).

\bibitem[KM98]{KM98}
J. Koll{\'{a}}r and S. Mori, \textit{Birational geometry of algebraic varieties}, Cambridge Tracts in Mathematics, vol. 134, Cambridge University Press, Cambridge, 1998.

\bibitem[KN74]{KN74}
L. Kuipers and H. Niederreiter,
\emph{Uniform distribution of sequences}, xiv+390pp, Wiley-Interscience [John Wiley \& Sons], New
              York-London-Sydney, (1974).

\bibitem[Laz04]{Laz04}
R. Lazarsfeld:
Positivity in algebraic geometry. II Positivity for vector bundles, and multiplier ideals
Ergeb. Math. Grenzgeb.,
vol. 49,
Springer-Verlag, Berlin.

%\bibitem[LN18]{LN18}
%Y. Lee, and N. Nakayama, \emph{Grothendieck Duality and $\Q$-Gorenstein Morphisms}, Publ. Res. Inst. Math. Sci. \textbf{54} (2018), no. 3, 517--648.

%\bibitem[PS14]{PS14}
%Z. Patakfalvi and K. Schwede: \textit{Depth of $F$-singularities and base change of relative canonical sheaves},
%J. Inst. Math. Jussieu.
%\textbf{13}
%(2014),
%no. 1,
%43--63.

\bibitem[Smi97]{Smi97}
K. Smith,
\emph{$F$-rational rings have rational singularities},
Amer. J. Math.
\textbf{119}
(1997),
no. 1,
159--180.

\bibitem[Stacks]{Stacks}
A. J. de Jong et al.: \textit{The Stacks Project}, Available at \url{http://stacks.math.columbia.edu}.

\bibitem[SS10]{SS10}
K. Schwede and K. Smith, \emph{Globally {$F$}-regular and log {F}ano varieties}, Adv. Math. \textbf{224} (2010), no. 3, 863--894.

\bibitem[Tak04]{Tak04}
S. Takagi, \emph{An interpretation of multiplier ideals via tight closure}, J. Algebraic Geom. \textbf{13} (2004), no. 2, 393--415.

\bibitem[Xu20]{Xu20} C. Xu, \textit{A minimizing valuation is quasi-monomial}, Ann. of Math. (2), \textbf{191} (2020), no. 3, 1003--1030.

\bibitem[Xu25]{Xu25}
C. Xu, \emph{K-stability of {F}ano varieties}, 50. xi+411pp. Cambridge University Press, Cambridge, (2025).

\end{thebibliography}
\end{document}